\documentclass[11pt]{amsart}
\usepackage{amssymb,tikz}
\usepackage{subfigure}
\usepackage{amsmath,bm}
\usepackage{amsthm}
\usepackage{amsbsy}
\usepackage{psfrag}
\usepackage{pstricks}
\usepackage{graphics}
\usepackage{graphicx}
\usepackage{enumitem}
\theoremstyle{plain}
\newtheorem{theorem}{Theorem}[section]
\newtheorem{lemma}[theorem]{Lemma}

\newtheorem{Definition}[theorem]{Definition}

\newtheorem{Problem}[theorem]{Problem}
\theoremstyle{remark}
\newtheorem{remark}[theorem]{Remark}

\newcounter{example}[section]
\newenvironment{example}[1][]{\refstepcounter{example}\par\medskip
	\textbf{Example~\thesection.\theexample. #1} \rmfamily}{\medskip}
\usepackage{cleveref}
\DeclareUnicodeCharacter{0308}{HERE!HERE!}
\begin{document}
\allowdisplaybreaks[4]
\numberwithin{figure}{section}
\numberwithin{table}{section}
 \numberwithin{equation}{section}
% \numberwithin{figure}{section}
%
\title[$\bm{L}^{\infty}$- error estimators of quadratic FEM for the Signorini problem]
 {Supremum norm A Posteriori Error control of Quadratic Finite Element Method for the Signorini problem}
  \author{Rohit Khandelwal}
 \address{Department of Mathematics, Indian Institute of Technology Delhi - 110016}
 \email{rohitkhandelwal004@gmail.com }
 \author{Kamana Porwal}\thanks{The second author's work is supported  by CSIR Extramural Research Grant}
 \address{Department of Mathematics, Indian Institute of Technology Delhi - 110016}
 \email{kamana@maths.iitd.ac.in}
 \author{Tanvi Wadhawan}\thanks{}
 \address{Department of Mathematics, Indian Institute of Technology Delhi - 110016}
 \email{tanviwadhawan1234@gmail.com}
\date{}

\begin{abstract}
In this paper, we develop a new residual-based pointwise a posteriori error estimator of the quadratic finite element method for the Signorini problem. The supremum norm a posteriori error estimates enable us to locate the singularities locally to control the pointwise errors. In the analysis the discrete counterpart of contact force density is constructed suitably to exhibit the desired sign property. We employ a priori estimates for the standard Green's matrix for the divergence type operator and introduce the upper and lower barriers functions by appropriately modifying the discrete solution. Finally, we present numerical experiments that  illustrate the excellent performance of the proposed error estimator.
\end{abstract}
\keywords{A posteriori error analysis, contact problems, supremum norm, variational inequalities, density forces, barrier functions, Green's matrix}
\subjclass{65N30, 65N15}
\maketitle
%%%%%%%%%%%%%%%%%%%%%%%%%%%%%%%%%%%%%%%%%%%%%%%%%%%%%%%%%%%%%%%%
\allowdisplaybreaks
\def\R{\mathbb{R}}
\def\cA{\mathcal{A}}
\def\cK{\mathcal{K}}
\def\cN{\mathcal{N}}
\def\p{\partial}
\def\O{\Omega}
\def\bbP{\mathbb{P}}
\def\cV{\mathcal{V}}
\def\cV{\mathcal{N}}
\def\cM{\mathcal{M}}
\def\cT{\mathcal{T}}
\def\cE{\mathcal{E}}
\def \cW{\mathcal{W}}
\def \cJ{\mathcal{J}}
\def \cV{\mathcal{V}}
\def\bF{\mathbb{F}}
\def \cW{\mathcal{F}}
\def\bC{\mathbb{C}}
\def\bN{\mathbb{N}}
\def\ssT{{\scriptscriptstyle T}}
\def\HT{{H^2(\O,\cT_h)}}
\def\mean#1{\left\{\hskip -5pt\left\{#1\right\}\hskip -5pt\right\}}
\def\jump#1{\left[\hskip -3.5pt\left[#1\right]\hskip -3.5pt\right]}
\def\smean#1{\{\hskip -3pt\{#1\}\hskip -3pt\}}
\def\sjump#1{[\hskip -1.5pt[#1]\hskip -1.5pt]}
\def\jumptwo{\jump{\frac{\p^2 u_h}{\p n^2}}}

\noindent
%{\textcolor{blue} {New changes are made in  blue colour}}.\\
%{\textcolor{green} {Reviewer 1: Changes are made in  green colour}}.\\
%{\textcolor{magenta} {Reviewer 2: Changes are made in  magenta colour}}.

%%%%%%%%%%%%%%%%%%%%%%%%%%%%%%%%%%%%%%%%%%%%%%%%%%%%%%%

%%%%%%%%%%%%%%%%%%%%%%%%%%%%%%%%%%%%%%%%%%%%%%%%%%%
\section{Introduction} \label{sec1}
\noindent
Let $\Omega \subset \mathbb{R}^2$ denotes an elastic body with Lipschitz boundary $\partial{\Omega}$ which is partitioned into three non overlapping mutually disjoint sets $\partial \Omega = {\bar\Gamma_D} \cup {\bar\Gamma_N} \cup {\bar\Gamma_C}$, where $\Gamma_D$ is the Dirichlet boundary with $meas(\Gamma_D) > 0$, $\Gamma_C$ and $\Gamma_N$ are the contact and Neumann boundaries, respectively. {Here, $\Gamma_D,\Gamma_N$ and $\Gamma_C$ are open subsets of $\partial \Omega$.} In this
article, we consider the model Signorini (unilateral contact) problem whose strong form is to find the displacement vector $\bm{u}: \Omega  \longrightarrow \mathbb{R}^2$ such that 
\begin{align} 
-\mbox{div}\bm{\sigma}\bm{(u)}  &= \bm{f} \hspace{1.4cm}\mbox{in}~ \Omega, \label{CP}\\ \bm{\hat{\sigma}(u)} &= \bm{g} \hspace{1.34cm}\mbox{on}~ \Gamma_N, \label{eq:NC}\\  \bm{u} &= \bm{0} \hspace{1.41cm} \mbox{on}~ \Gamma_D, \label{eq:DP}\\ u_n \leq \chi ,\hspace{0.3cm} \hat{\sigma}_n(\bm{u}) \leq 0,\hspace{0.2cm}  (u_{n}-\chi) \hat{\sigma}_n(\bm{u}) &=0 \hspace{1.49cm} \mbox{on}~ \Gamma_C, \\ \bm{\hat{\sigma}}_{\tau}(\bm{u}) &=0 \hspace{1.49cm}\mbox{on}~ \Gamma_C, \label{eq:FC}
\end{align} where $\chi: \Gamma_C \rightarrow \mathbb{R}$ denotes the gap function representing the distance between $\Omega$ and rigid obstacle. Further $\bm{g} \in [L^{\infty}(\Gamma_N)]^2$ denotes the surface force and $\bm{f} \in [L^{\infty}(\Omega)]^2$ be the volume force density.  For a matrix valued function $\bm{A}= (a_{ij}) \in \mathbb{R}^{2\times 2}$, its divergence is defined as
\begin{align*}
(\text{div}(A))_i = \sum_{j} \frac{\partial}{\partial x_{j}} (a_{ij}) \quad 1\leq i \leq 2.
\end{align*}
{In this article}, vector valued functions are denoted by bold symbols and the scalar valued functions are written in the usual way. Let
\begin{align} 
\bm{\epsilon(u)}:= \frac{1}{2} (\nabla \bm{u}^T + \nabla \bm{u}) ~~\text{and}~~
\bm{\sigma(u)}:= 2\mu \bm{\epsilon(u)} + \zeta~(tr\bm{\epsilon(u)}) \bm{I}\label{strain}
\end{align}
be the linearized strain tensor and stress tensor, respectively, where $\bm{I}$ is an identity matrix of order $2$ and $\mu >0$, $\zeta >0$ are the Lam$\acute{e}$ constants which are expressed using Young's modulus $E$ and Poisson ratio $\nu$ via \cite{kikuchi1988contact,walloth2012}
\begin{align} \label{Lame}
\zeta = \frac{E\nu}{(1+\nu)(1-2\nu)},~~\mu=\frac{E}{2(1+\nu)}.
\end{align}
In order to avoid working with the space $H^{1/2}_{00}(\Gamma_C)$, we assume ${\bar\Gamma_C} \cap {\bar\Gamma_D} = \emptyset$ (see \cite{kikuchi1988contact}).
Further, we denote $\bm{u}$ as $\bm{u}= u_i \bm{e}_i$ where \{$\bm{e}_i; i=1,2$\} are the standard basis vectors of $\mathbb{R}^2$.   Let $\bm{n}$ denotes the outward unit normal vector to $\partial{\Omega}$.  The (linearized) non-penetration condition $u_n \leq \chi$ where $u_n:= \bm{u} \cdot \bm{n}$,  arises due to the contact of two solid bodies. It incites the contact stresses $\bm{\hat{\sigma}(u)} := \bm{\sigma(u)} \bm{n}$ in the direction of the normal at $\Gamma_C$. The complimentarity condition on $\Gamma_C$ is given by 
\begin{align*}
(u_{n}-\chi) \hat{\sigma}_n(\mathbf{u}) &=0,
\end{align*}
where we use the notation $\hat{\sigma}_n(\mathbf{u}):= \bm{\hat{\sigma}(u)} \cdot \bm{n}$ for the contact stresses. It is clear that $\hat{\sigma}_n(\mathbf{u})=0$ provided there is no contact. We assume that there is no frictional effects on $\Gamma_C$  i.e., the tangential boundary stresses $\bm{\hat{\sigma}}_{\tau}(\mathbf{u}) :=  \bm{\hat{\sigma}(u)} - \hat{\sigma}_n(\mathbf{u}) \bm{n} $ are assumed to be zero. In {the} analysis, for any Banach space/Hilbert space $H$,  we use the notation $\bm{H}=H \times H$ to describe the space of vector valued functions.

\begin{remark}
The unilateral contact problem has several relevant applications in the physical and mechanical sciences. For example, we consider the model problem from deformable solid mechanics. The body $\Omega$ is subjected to the external forces $\bm{f}$ and $\bm{g}|_{\Gamma_N}$ and it is further supported by the frictionless rigid membrane $\Gamma_C$ (see \Cref{FIg1}). The displacement of the domain $\Omega$ satisfies \eqref{CP}--\eqref{eq:FC} \cite{hild2002quadratic}.
% \begin{align} 
%-\mbox{div}\bm{\sigma}\bm{(u)}  &= \bm{f} \hspace{1.4cm}\mbox{in}~ \Omega,\\ \bm{\hat{\sigma}(u)} &= \bm{g} \hspace{1.34cm}\mbox{on}~ \Gamma_N,\\  \bm{u} &= \bm{0} \hspace{1.41cm} \mbox{on}~ \Gamma_D.
%\end{align} 
Another example comes from the hydrostatics,  consider a fluid which is contained in a semi-permeable domain $\Omega$ which permits the fluid to travel through only in one direction and assume the body $\Omega$ is partially bounded by a membrane $\Gamma_C$. If the external pressure  $\chi|_{\Gamma_C}$ is applied, then the resulting internal pressure satisfy equations \eqref{CP}--\eqref{eq:FC}.
\setcounter{figure}{0}
\renewcommand{\thefigure}{\arabic{figure}}
\begin{figure}[ht!] 
	\begin{center}
		\includegraphics[height=10cm,width=15cm]{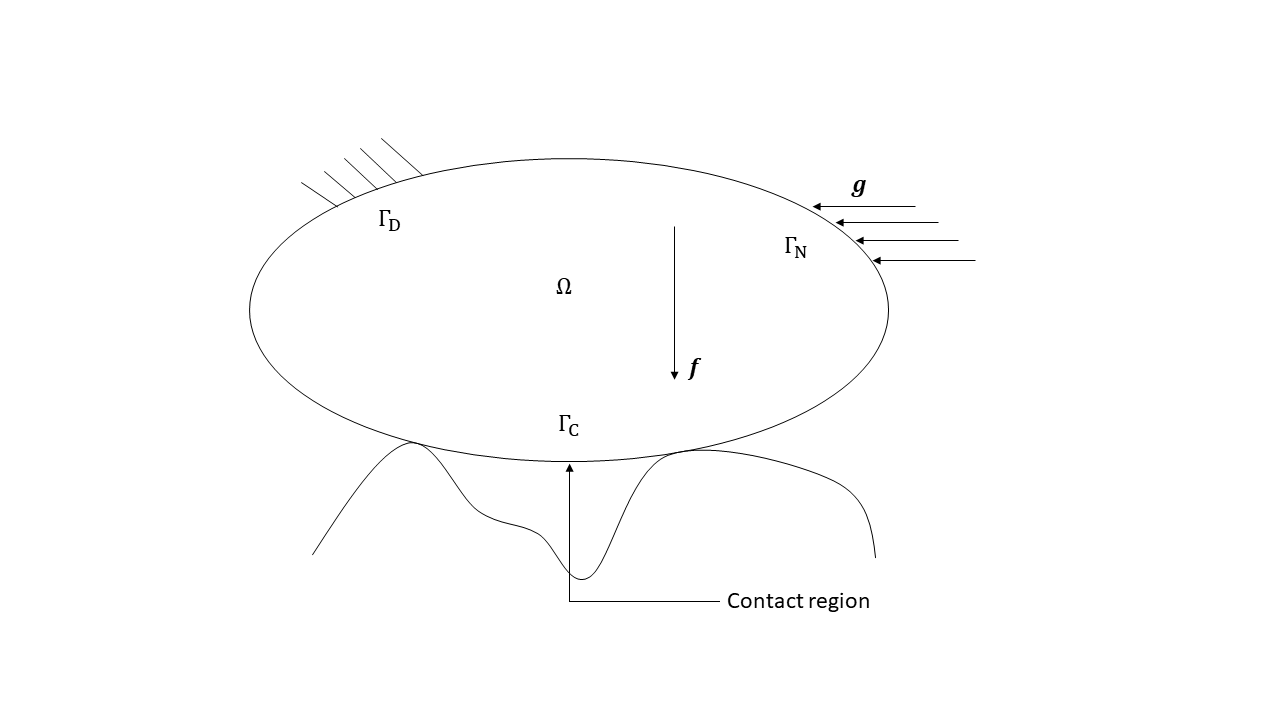}
		\caption{Model Problem}
		\label{FIg1}
	\end{center}
\end{figure}
\end{remark}
\noindent 
The Signorini problem is a prototype of the elliptic variational inequality of the first kind. It's continuous variational formulation reads: {to} find $\bm{u}=(u_1,u_2) \in \bm{\cK} \subseteq \bm{V} $ such that
\begin{align}\label{eq:CVI}
a(\bm{u},\bm{v}-\bm{u})\geq L(\bm{v-u})  \quad \forall~ \bm{v}\in \bm{\cK},
\end{align} 
where
\begin{enumerate}
	\item[a)] $\bm{\cK}=\{\bm{v}\in \bm{V}~|~v_n \leq \chi~~ \mbox{on}~~ \Gamma_C \}$ with $\bm{V} =\bm{H}^1_{\Gamma_D}(\Omega)= \{\bm{v} \in \bm{H}^1(\Omega)~|~ \bm{v}= \bm{0} ~~\text{on}~ \Gamma_D\}$,
	\item[c)] $a(\bm{z},\bm{v}) = \int_{\Omega} \bm{\sigma}(\bm{z}): \bm{\epsilon}(\bm{v})dx, \quad \forall \bm{z},\bm{v} \in \bm{V}$,
	\item[d)] $ L(\bm{v}) =(\bm{f},\bm{v}) + \langle \bm{g},\bm{v} \rangle_{\Gamma_N} \quad \forall \bm{v} \in \bm{V}$,
	\end{enumerate}
	here $( \cdot,\cdot)$ denotes $[L^2(\Omega)]^2$ inner-product.
	In the subsequent analysis,  $\langle \cdot, \cdot \rangle_{-1,1}$ denotes the duality pairing between the space $\bm{V}$ and its dual $\bm{V}^*$.  The corresponding norms on the space  $\bm{V}$ and $\bm{V}^*$ are given by $\|\cdot\|_{1}= \|\cdot\|_{\bm{H^1(\Omega)}}$ and $\|\cdot\|_{-1}= \|\cdot\|_{\bm{H^{-1}(\Omega)}}$, respectively.
\noindent
%\begin{remark}
	By the classical work of Stampacchia \cite{glowinski2008lectures}, we have the existence and uniqueness of the solution for variational inequality \eqref{eq:CVI}.
	For the ease of presentation, we assume $\bm{n}=(1,0)$ to be an outward unit normal to $\Gamma_C$, hence we rewrite the discrete set $\bm{\cK}$ as
	\begin{align}
	\bm{\cK}=\{\bm{v}=(v_1,v_2)\in \bm{V}~|~v_1 \leq \chi~~ \mbox{on}~~ \Gamma_C \}.
	\end{align}
%\end{remark}

	\begin{remark} The solution $\bm{u} \in \bm{H}^1_{\Gamma_D}(\Omega) \cap \bm{C}^{0,\alpha}(\bar{\Omega})$ for some $\alpha>0$ under the assumption that $\bm{f} \in \bm{L}^{\infty}(\Omega)$ and the gap function $\chi$ is H{\"o}lder continuous \cite{Signorini:1977,Caffarelli:1979}.
	\end{remark}

\noindent A priori error estimates for \eqref{eq:CVI} are discussed in \cite{brezzi1977error,scarpini1977error,kikuchi1988contact,glowinski2008lectures} using conforming linear {finite element method (FEM)}. Higher-order finite element methods \cite{ciarlet2002finite} might contribute in deriving more accurate discrete solutions. We refer to the articles \cite{hild2002quadratic,belhachmi2003quadratic} for a priori error estimates of quadratic finite element methods for \eqref{eq:CVI}. The article \cite{hild2002quadratic} exploited the mixed formulation in which the unknowns are the displacement $\bm{u}$ and the contact pressure, and the article \cite{belhachmi2003quadratic} described two nonconforming quadratic approximations corresponding to the Signorini problem. There has been enormous activity in recent years for {developing} a posteriori error estimates of finite element methods for {the} variational inequality \eqref{eq:CVI}. We can obtain robust a posteriori error estimates for the elliptic variational inequalities by associating the true error in constraining density forces in the error measure. This crucial observation was first made by Veeser \cite{veeser2001efficient} for the obstacle problem while deriving a posteriori error bounds in the energy norm using the conforming finite element method. We refer the articles \cite{weiss2009posteriori,krause2015efficient,gudi2016posteriori,walloth2019reliable} for a posteriori error analysis for the Signorini problem using linear FEM. {In the article \cite{KP:2022:QuadSignorini}, authors have developed a residual-based energy norm a posteriori error estimator using the quadratic conforming finite element method for the frictionless unilateral contact problem. The analysis developed in all the articles mentioned previously is on the energy or Sobolev space norms.} In this article, we propose and derive a posteriori error analysis for \eqref{eq:CVI} in the supremum ($\bm{L}^{\infty}$) norm using quadratic conforming FEM. To the best of the author's knowledge, this paper is the first attempt in this direction. A posteriori error estimates in the supremum norm for the variational inequalities capture the discrete solution's pointwise accuracy and provide more localized knowledge about the approximation.

\noindent
In the past decade, {several} works \cite{nochetto1995pointwise,dari1999maximum,demlow2012pointwise} discussed pointwise a posteriori error estimates for the linear elliptic  problems. In the articles \cite{nochetto1995pointwise,dari1999maximum} and \cite{demlow2012pointwise}, the analysis is carried out using conforming and discontinuous Galerkin (DG) linear finite element methods, respectively. {Supremum norm} a priori error estimates for the elliptic variational inequalities were initially derived by Nitsche \cite{nitsche1977convergence} and Baiocchi \cite{baiocchi1977estimations}. The authors in \cite{nochetto2003pointwise,nochetto2005fully,BGP:2022:Obstacle} have considered linear finite element method and derived the reliable and efficient a posteriori error estimates for the elliptic obstacle problem in the supremum ($\bm{L}^{\infty}$) norm. {Recently, in \cite{khandelwal2022pointwiseq}, the pointwise {\it{a posteriori}} error estimates are developed for the obstacle problem using quadratic conforming FEM. The analysis in the article \cite{khandelwal2022pointwiseq} uses constraints only at the midpoints of the edges of triangulation and the Lagrange multiplier constructed suitably to achieve the optimal order of convergence. } In \cite{KP:2021:Signorini}, the pointwise adapative FEM is studied for the Signorini problem using linear conforming elements. The proof in \cite{KP:2021:Signorini} is based on the direct use of the bounds and a priori estimates of the Green's matrix for the divergence type operator \cite{dolzmann1995estimates}. If no contact occurs, the proposed error estimator in the article \cite{KP:2021:Signorini} reduces
to the standard error estimator for the linear elasticity \cite{walloth2012}.

\noindent
 In this work, the piece-wise quadratic discrete space $\bm{V}_h$ is decomposed {carefully} in order to obtain the desired properties for the quasi discrete contact force density \cite{walloth2012,moon2007posteriori}. The sign property (Lemma \ref{sign2}) of quasi discrete contact force density helps crucially in deriving the reliability estimates {for the proposed a posteriori error estimator} in the supremum norm. Moreover, we introduce the dual problem with the aid of {a} corrector function \cite{nochetto2003pointwise,nochetto2005fully},  upper and lower barrier functions corresponding to the continuous solution $\bm{u}.$
%  The reliability and the efficiency of the proposed error estimator is discussed. 
  The proof of the reliability estimates  hinges mainly on the appropriate construction of the discrete contact force density,  a priori error estimates for the Green's matrix of divergence type operator \cite{dolzmann1995estimates} and the pointwise estimate on the corrector function. Our analysis is slightly different from the articles \cite{nochetto2003pointwise,nochetto2005fully,nochetto2006pointwise} which relies on the regularized Green's function in order to derive the pointwise a posteriori error estimates for the obstacle problem. To derive the local efficiency estimates, we followed the approach shown in the articles \cite{moon2007posteriori,krause2015efficient} and defined the quasi discrete contact force density differently on the distinct parts of $\Gamma_C$.\\

\noindent 
The rest of the paper is discussed as follows: In Section \ref{sec2}, we define the continuous contact force density and discuss related results. Some standard regularity results and a priori estimates for the Green's matrix have been introduced in {Section \ref{sec2}. We state the discrete formulation of the continuous variational inequality and define some associated notations  in Section \ref{sec3}.} The discrete analogous of the contact force density is discussed in Section \ref{sec4}. In Section \ref{sec5}, we define a continuous linear functional on $\bm{V}$, called the quasi discrete contact force density, unlike the article \cite{gudi2016posteriori}, where the authors defined the discrete Lagrange multiplier which is indeed a linear functional on $\bm{H}^{\frac{1}{2}}(\Gamma_C)$. The {main contributions of the paper}, i.e., reliability and efficiency of the proposed a  posteriori error estimator, are discussed in Section \ref{sec6}. In Section \ref{sec7}, we present several numerical experiments  demonstrating the performance of a posteriori error estimator.

\section{Continuous contact force denisty and Green's  matrix} \label{sec2}
\noindent In this section, we introduce the continuous contact force density which is the residual of the displacement $\bm{u}$ with respect to the continuous variational inequality \eqref{eq:CVI} and Green's matrix associated to a  divergence type operator. We first {recall} the subdifferential of a proper functional.
 \begin{Definition}
	Let $X$ be a Hilbert space and $m: X \rightarrow \mathbb{R} \cup \{-\infty, \infty\}$ be a proper map, i.e., $m(x) > - \infty \forall~ x \in X$ and $m \neq + \infty$. Let $x \in X$, then the subdifferential of $m$ in $x$ is defined by
	\begin{align*}
	\nabla_{sub} m(x):= \{x^* \in X^*~|~m(y)-m(x) \geq {x}^*(y-x) \quad \forall~ y \in X\},
	\end{align*}
	where $X^*$ denotes the dual of $X$.
\end{Definition}
\noindent
Now, we define the continuous contact force density $\bm{\lambda}=(\lambda_1,\lambda_2) \in  \bm{V}^*$ in the following way
\begin{align}\label{eq:sigmadef}
\langle \bm{\lambda}, \bm{v}\rangle_{-1,1} =L(\bm{v}) -a(\bm{u},\bm{v})\quad \forall~ \bm{v}\in \bm{V}.
\end{align}
\begin{remark}
$\bm{\lambda}$ can be treated as an element of $\nabla_{sub} I_{\bm{\cK}}(\bm{u})$ \cite{walloth2012}, where $I_{\bm{\cK}}$ is the indicator function \cite{glowinski1980numerical} defined by
	\[ 
 I_{\bm{\cK}}(\bm{v})=
\begin{cases}
0 & \text{if}~ \bm{v} \in \bm{\cK}, \\
+\infty &  \text{otherwise},
\end{cases} 
\] and we note that $\lambda_1 \in \nabla_{sub} I_{\bm{\cK}}({u}_1)$ where
	\[ 
\nabla_{sub} I_{\bm{\cK}}(u_1)=
\begin{cases}
0 & \text{if}~ u_1 < \chi, \\
[0,\infty) & \text{if}~ u_1 = \chi.
\end{cases} 
\]
\end{remark}
\begin{remark}
	Let $supp(v)$ denotes the support of the function $v$, then we deduce 
	\begin{align} \label{eq:IUY}
	supp(\lambda_1) &\subset \{u_1=\chi\},
	\end{align}
	from {the} definitions of subdifferential of an indicator function and continuous contact force density $\bm{\lambda}$ \cite{walloth2012}.
\end{remark} 
%\begin{remark}
%	In our analysis, the continuous contact force density is a continuous  linear functional on $\bm{V}$, in comparison to \cite{gudi2016posteriori}, where the continuous Lagrange multiplier is a linear functional on $\bm{H}^{\frac{1}{2}}(\Gamma_C)$.
%\end{remark}
%\begin{remark}
%In order to state the strong form of the Signorini contact problem, we assume that $\bm{u}$ is sufficiently regular. The following problem is well known in the literature \cite{}. We seek $\bm{u}: \bar{\Omega} \rightarrow \mathbb{R}^2$ such that the following holds:
%\end{remark}
\noindent The following results \cite[Lemma 2.7]{KP:2021:Signorini}  can be realized by a use of \eqref{eq:sigmadef}, \eqref{eq:CVI} and a suitable choice of the test function in \eqref{eq:sigmadef}.
%will be crucial in the forthcoming analysis.
\begin{lemma} \label{Lemmaa}
	It holds that 
	\begin{align}
	\langle \bm{\lambda}, \bm{u}-\bm{v}\rangle_{-1,1} &\geq 0\quad \forall~ \bm{v} \in \bm{\cK}, \label{eq:SCT} \\ 
	\langle \bm{\lambda} , \bm{\phi} \rangle_{-1,1} &\geq 0 \quad \forall~ \bm{0} \leq \bm{\phi} \in \bm{V} . \label{eq:POOI}
	\end{align}
\end{lemma}
%\begin{proof}
%We begin to show $(\ref{eq:SCT})$. Let $\bm{v} \in \bm{\cK}$ and $\bm{u} \in \bm{\cK} \subset \bm{V}$ be the continuous solution, then the following holds
%	\begin{align*}
%	\langle \bm{\lambda} , \bm{v}-\bm{u} \rangle_{-1,1}= L(\bm{v}-\bm{u}) -a(\bm{u},\bm{v}-\bm{u}),
%	\end{align*}therefore using \eqref{eq:CVI}, we have the proof. In order to prove \eqref{eq:POOI}, let $\bm{\phi} \in \bm{V}$ and $\bm{\phi} \geq \bm{0}$ and we choose $\bm{v}= \bm{u - \phi} \in \bm{\cK}$ in equation \eqref{eq:CVI} to get $ \langle \bm{\lambda} , \bm{\phi} \rangle_{-1,1} \geq  0$.
%\end{proof}
\noindent
We collect a key representation for $\bm{\lambda}$ in the next lemma and refer the article \cite{KP:2021:Signorini} for the details.
\begin{lemma}
	It holds that
\begin{align} \label{repre}
\langle \bm{\lambda}, \bm{v} \rangle_{-1,1} =  \langle {\lambda}_1, {v_1} \rangle_{-1,1},
\end{align} 
where $\bm{v}=(v_1,v_2) \in \bm{V}.$
\end{lemma}
\noindent
{Next, we introduce the standard Green's matrix for divergence type operators as it plays a key role in performing the supremum norm a posteriori error analysis of associated non-linear problems. For definition and regularity results of the Green's matrix for unconstrained problem, we refer the reader to the articles \cite{dolzmann1995estimates,hofmann2007green,dong2009green}.
\subsection{Green's matrix}
\noindent
We establish this subsection by restating $\bm{\epsilon(u)}$ and $\bm{\sigma(u)}$ (defined in \eqref{strain}) in a different way as follows. Define
\begin{align*}
\epsilon^i_{\alpha}(u)&:= \frac{1}{2} \Big( \frac{\partial u_i}{\partial x_{\alpha}} + \frac{\partial u_{\alpha}}{\partial x_i}\Big), \\
\sigma_{\alpha}^i(u) &:= a_{\alpha\beta}^{ij} \frac{\partial u_{\beta}}{\partial x_j},
\end{align*}
where, the Lam$\acute{e}$ operator $a_{\alpha\beta}^{ij}$ is defined by
\begin{align*}
a_{\alpha\beta}^{ij}= \mu (\delta_{ij} \delta_{\alpha \beta}+ \delta_{j \alpha}\delta_{i \beta}) +\zeta  \delta_{i \alpha} \delta_{j \beta},
\end{align*}
with $\delta_{i,j}  ~\text{denoting the Kronecker's Delta i.e.,}~\delta_{i,j}=1~\text{if}~i=j~\text{and}~0~\text{elsewhere}$, $\mu$ and $\zeta $ are defined in equation (\ref{Lame}). In the analysis, we denote $\bm{x}'$ to be the transpose of vector $\bm{x}.$ Further, we formulate the divergence type operator in the following way
\begin{align} \label{dope}
\mbox{div}\bm{\sigma}\mathbf{(u)}:= (Fu)^{\alpha}= \sum_{i,j=1}^{2} \sum_{\beta =1}^{2} \frac{\partial}{\partial x_i} a_{\alpha \beta}^{ij} \frac{\partial u_{\beta}}{\partial x_j}, \quad \alpha= 1,2.
\end{align}
The existence of the Green's matrix for operator $F$ is stated in the next lemma. We refer the articles \cite{dolzmann1995estimates,hofmann2007green} for the detailed readings. 
\begin{lemma} \label{regularity}
	Let $\delta_z$ be a Dirac delta function having a unit mass at $z \in \Omega$ and let $F$ be the operator defined in \eqref{dope} which satisfies
	\begin{equation*}
	(F \bm{u})^i = \sum_{j=1}^{2} F_{ij}u_j:= \sum_{j=1}^{2} \sum_{\alpha, \beta =1}^{2} D_{\alpha} (a_{ij}^{\alpha \beta}(x)D_{\beta}u_j),~~i=1,2,
	\end{equation*}
	where $a_{ij}^{\alpha \beta}(x)$ satisfies the following two conditions: $\exists$ positive constants $ M $ and $C$ such that 
	\begin{align*}
	a_{ij}^{\alpha \beta}(x) \phi_{\beta}^j \phi_{\alpha}^i &\geq M |\bm{\phi}|^2~~\forall~ x \in \Omega,\\
	\sum_{n,m=1}^{2} \sum_{\alpha, \beta=1}^{2} |a_{nm}^{\alpha \beta}(x)|^2 &\leq C ~~ \forall~ x \in \Omega,
	\end{align*} (see \cite{dolzmann1995estimates})
	then, there exists a Green's (fundamental) matrix $ \bm{G}^z:= (G_{nm}(\cdot, z))_{n,m=1}^2$ (defined on the domain $\{(x,y) \in \Omega \times \Omega: x\neq y \}$) satisfying the following equations in the sense of distributions
	\begin{align} \label{eq:GR}
	-\sum_{j=1}^{2} F_{ij}G_{jk}(\cdot, z) &= \delta_{ik}\delta_z(\cdot) \quad \text{in}~~ \Omega, \\ \bm{G}^z &= 0 \quad \hspace{1cm}\text{on}~\Gamma_D , \\ a_{\alpha \beta}^{ij} D_{\beta} G_{jk}(\cdot,z) n_{\alpha} & =0 \quad \hspace{1cm} \text{on}~ \Gamma_N \cup \Gamma_C. \notag
	\end{align}
	Also, for any $1 \leq s < 2$,
	\begin{align} \label{bound4}
	\|\bm{G}^z\|_{\bm{W^{1,s}_0(\Omega)}} \leq C.
	\end{align}
	Lastly, for $(x,z) \in (\Omega \times \Omega, x \neq z)$, there holds,
	%	\begin{align} \label{eq:BGR}
	%	|\bm{G}(x,z)| \leq C ~|x-z|^{-1}, \quad d=3.
	%	\end{align}
	\begin{align} \label{eq:BGRR}
	|\bm{G}(x,z)| &\lesssim~ \text{ln} \frac{C}{|x-z|},
	\end{align}
\end{lemma}
\noindent
where $X \lesssim Y$ denotes $X \leq CY$. Here, $C$ denotes a positive constant independent of the mesh parameter $h$.} 
%\begin{proof}
%Let
%	\begin{align*}
%	\bm{V_0}= \{ \bm{v}=(v_1,v_2) \in \bm{V}, ~\bm{v}\cdot \bm{n}=v_1 = 0~\text{on}~\Gamma_C\}.
%	\end{align*}
%	Since $\bm{v}= \bm{u} \pm \tilde{\bm{v}} \in \bm{\cK}~\forall~\tilde{\bm{v}} ~\in\bm{V_0}$, therefore inequality \eqref{eq:CVI} reduces to
%	\begin{align}\label{2.6}
%	a(\bm{u},\bm{\tilde v})~=~L(\tilde{\bm{v}})~\forall~\tilde{\bm{v}}\in \bm{V_0}.
%	\end{align}
%		We have the decomposition $\bm{v} = (v_1,v_2):= \bm{w_1}+\bm{w_2}$ where $\bm{w_1}=(v_1,0)$ and $\bm{w_2}=(0,v_2)$, using equation \eqref{eq:sigmadef}, we can rewrite $\langle \bm{\lambda},\bm{v} \rangle_{-1,1} = \langle {\lambda}_1, {v_1} \rangle_{-1,1} +\langle {\sigma_2},{v_2} \rangle_{-1,1} $, where
%		\begin{align*}
%		\langle {\lambda}_1,{v_1} \rangle_{-1,1} &= L(\bm{w_1}) - a (\bm{u}, \bm{w_1}), \\
%		\langle {\lambda}_2,{v_2} \rangle_{-1,1} &= L(\bm{w_2}) - a (\bm{u}, \bm{w_2}). 
%		\end{align*}
%		As $\bm{w_2} \in \bm{V_0}$,  using \eqref{2.6}, we obtain
%		\begin{align*}
%		\langle {\lambda}_2,{v_2} \rangle_{-1,1} &= 0.
%		\end{align*}
%	\noindent
%	Thus, we complete the proof.
%\end{proof}
\section{Preliminaries and discrete variational inequality}\label{sec3}
\par
\noindent
In this section, we introduce some preliminary notations and results which will be used in later analysis. The continuous piece-wise quadratic finite element space is used for the discrete approximation of the continuous space $\bm{V}$. We assume that the triangulation $\cT_h$ is regular \cite{ciarlet2002finite} by means that there are no hanging nodes in $\cT_h$ and the elements in $\cT_h$ are shape-regular. Furthermore, the elements in $\cT_h$ are assumed to be closed.  We denote {by} $\mathbb{P}_k(T)$  to be the set of all polynomials of degree at most $k, 0 \leq k \in \mathbb{Z}$ over the triangle $T \in \cT_h$. We set $h_T =\text{diameter of } T$, $h_{min}:= \text{min}\{h_T~;~T \in \cT_h\}$$~\text{and}~h:=\max\{h_T : T\in\cT_h\}$. Let $|T|$ denote the area of {the element} $T$ and $\cN_h$ is the set of all nodes of the triangulation $\cT_h$. The set of interior nodes is denoted by $\cN_h^i$ and based on the distinct boundaries, we let $\cN_h^D$ to be the set of nodes on $\Gamma_D$. The set of nodes on $\Gamma_N$ is denoted by $\cN_h^N$ and $\cN_h^{\bar{N}}$ is the collection of all nodes on the closure of Neumann boundary. $\cN_h^C$ is the set of nodes on $\Gamma_C$. We consider $\cN_h^0$ to be set of nodes in $\cT_h$ that lies on
$\cN_h \setminus \cN_h^D$. Let us denote by $\cN^T_h$ the set of all vertices of the element $T$. The set of all edges of $\cT_h$ is denoted by $\cE_h$ and interior edges by $\cE_h^i$. We assume $h_e$ to be the length of an edge $ e\in\cE_h$ and moreover, we classify the set $\cM_h$ which is the set of all midpoints of edges of $\cT_h$ in the following way
\begin{align*}
\cM^e_h &= {\text{midpoint of the edge}~e},\\
\cM_h^i&=\text{set of all midpoints for } e \in \cE^i_h,\\
\cM_h^D&=\text{set of all midpoints on } \Gamma_D,\\
\cM_h^C&=\text{set of all midpoints lying on } \Gamma_C,\\
\cM_h^0&=\text{set of all midpoints in } \cT_h \text{ that are in }
\cM_h \setminus \cM_h^D,\\
\cM^T_h&=\text{set of all midpoints of edges of {the element} }T. 
\end{align*}
Let $\Omega_p$ denote the union of all elements sharing the node $p$ and $h_p$ is the diameter of $\Omega_p$. Lastly, we collect the set of edges corresponding to distinct boundaries as
\begin{align*}
\Gamma_{p,I}&= \text{union of edges in the interior of}~\Omega_p~\text{excluding the boundary of }\Omega_p,\\
\Gamma_{p,C}&= \Gamma_C \cap \partial\Omega_p,\\
\Gamma_{p,D}&= \Gamma_D \cap \partial\Omega_p,\\
\Gamma_{p,N}&= \Gamma_N \cap \partial\Omega_p.
\end{align*}
For $\Omega' \subset \Omega$, we set $\|\cdot\|_{m,p,\Omega'} = \|\cdot\|_{W^{m,p}(\Omega')}~\text{where}~m \in \mathbb{Z},~1 \leq p \leq \infty$. {Given a function $\bm{v}$, we define $\bm{v}^+= \text{max}\{\bm{v}, \bm{0}\}$ to be positive part of $\bm{v}$.}
\vspace{0.5 cm}
\par
\noindent
{
The conforming quadratic finite element space $\bm{V}_h$ is defined by
\begin{align*}
\bm{V}_{h}:= \{ \bm{v}_h \in [{C}^0(\bar{\Omega})]^2 ~|~ \bm{v}_h|_T \in [\mathbb{P}_2(T)]^2 ~\forall~ T \in \cT_h \}.
\end{align*}
Analogously,  we define space $\bm{V}^0_{h}$ by incorporating the essential boundary conditions as follows
\begin{align*}
\bm{V}^0_{h}:= \{ \bm{v}_h \in [{C}^0(\bar{\Omega})]^2 ~|~ \bm{v}_h|_T \in [\mathbb{P}_2(T)]^2 ~\forall~ T \in \cT_h~\text{and}~ \bm{v}_h= \bm{0}~\text{on}~ \Gamma_D \}.
\end{align*}}
Let $\{\phi_p\bm{e}_i,~p \in \cN_h \cup \cM_h, ~i=1,2\}$ be the nodal basis {functions} of $\bm{V}_h$.  { Consequently,   $\{\phi_p\bm{e}_i,~p \in \cN^0_h \cup \cM^0_h, ~i=1,2\}$ denotes the basis function for space $\bm{V}^0_{h}$ }. For any  $\bm{v}_h \in \bm{V}^0_h$, we {can} write
\begin{align} \label{eq:ES}
\bm{v}_h = \sum_{p \in \cN^0_h \cup \cM^0_h} \sum_{i=1}^{2} v_{h,i}(p) \phi_p \bm{e}_i,
\end{align}
using 	\[ 
\phi_p(z)=
\begin{cases}
1 & \text{if}~ z=p, \\
0 & \text{if}~ z \neq p,
\end{cases} ~~\quad z \in \mathcal{N}_h^0 \cup \mathcal{M}_h^0,
\] 
where $\bm{v}_h:= (v_{h,1},v_{h,2})$. Define
\begin{align}
\bm{{Z}}_h:= \text{span} \{\phi_p\bm{e}_i,~p \in (\cN_h^0 \cup \cM_h^0) \setminus  (\cN_h^C \cup \cM_h^C) \}.
\end{align}
Then, we have $\bm{\bm{V}}^0_h=\bm{\bm{Z}}_h \oplus \bm{\bm{Z}}^c_h$, where,
\begin{align}
\bm{Z}^c_h= \text{span} \{\phi_p\bm{e}_i,~p \in \cN_h^C \cup \cM_h^C \}
\end{align}
is the orthogonal complement of $\bm{Z}_h$ with respect to the inner product
\begin{align*}
  \langle {\bm{v}, \bm{w} \rangle_{\bm{V}^0_h} := \sum_{T\in \mathcal{T}_h} \frac{|T|}{3}\bigg(\sum_{z \in \cN^T_h} \bm{v}(z)\bm{w}(z) + \sum_{z \in \cM^T_h} \bm{v}(z)\bm{w}(z)}\bigg). 
\end{align*}
\noindent
\begin{Problem}
	Discrete variational inequality: We seek $\bm{u}_h \in \bm{\cK}_h$ such that
	\begin{align}\label{eq:DVI}
	a(\bm{u}_h,\bm{v}_h-\bm{u}_h)\geq L(\bm{v}_h-\bm{u}_h), \quad \forall~ \bm{v}_h\in \bm{\cK}_h
	\end{align}
	holds, where 
	\begin{align*}
	\bm{\cK}_h:=\{\bm{v}_h=(v_{h,1},v_{h,2})\in \bm{V}_h~|~v_{h,1}(p) \leq \chi_h(p)~\forall~ p \in\cN_h^C \cup \cM_h^C\}
	\end{align*}
	and $\chi_h$ denotes the {quadratic} Lagrange interpolation \cite{brenner2007mathematical} of $\chi$ on $\Gamma_C$.
\end{Problem}
	 \begin{remark}
	 	We observe that $\bm{\cK}_h \nsubseteq \bm{\cK}$ and the set $\bm{\cK}_h$ is closed, non empty and convex.
	 \end{remark}
 \noindent
The existence and uniqueness of solution of \eqref{eq:DVI} follows similarly as in the continuous case \cite{glowinski1980numerical}. In the next lemma, we collect some results related to the spaces $\bm{Z}_h$ and $\bm{Z}^c_h$.
\begin{lemma} \label{2.3}
	It holds that
	\begin{align} \label{char1}
	a(\bm{u}_h,\phi_p\bm{e}_i)-L(\phi_p\bm{e}_i)=0 \quad \forall p \in (\cN_h \cup \cM_h) \setminus  (\cN_h^C \cup \cM_h^C)~\text{and}~i=1,2
	\end{align}
	and 
	\begin{align} \label{char2}
	\begin{split}
	\begin{aligned}
	a(\bm{u}_h, {\phi_p\bm{e}_1})-L({\phi_p\bm{e}_1})~ &\leq 0,\\
	a(\bm{u}_h, {\phi_p\bm{e}_2})-L({\phi_p\bm{e}_2})~ &=0,
	\end{aligned} \qquad \forall~p \in \mathcal{N}_h^C \cup \mathcal{M}_h^C.
	\end{split}
	\end{align}
\end{lemma}
{
\begin{proof}
	Let $p \in (\cN_h \cup \cM_h) \setminus  (\cN_h^C \cup \cM_h^C) $. We obtain \eqref{char1} by choosing $\bm{v}_h= \bm{u}_h \pm \phi_p \bm{e}_i \in \bm{\cK}_h$ in \eqref{eq:DVI}. Next, for $p \in \mathcal{N}_h^C \cup \mathcal{M}_h^C$, we choose $\bm{v}_h =\bm{u}_h - \phi_p \bm{e}_1  \in \bm{\cK}_h$ and $\bm{v}_h =\bm{u}_h \pm \phi_p \bm{e}_2  \in \bm{\cK}_h$ in equation \eqref{eq:DVI} to derive the desired estimate \eqref{char2}.
\end{proof}}
%\begin{proof}
%	First, we prove \eqref{char1}. Note that for $i=1,2$, we have $\bm{v}_h= \bm{u}_h \pm \phi_p \bm{e}_i \in \bm{\cK}_h~\forall~p \in (\cN_h^0 \cup \cM_h^0) \setminus  (\cN_h^C \cup \cM_h^C).$ Hence, using \eqref{eq:DVI}, the proof of \eqref{char1} follows. Next, for $p~\in~(\mathcal{N}_h^C \cup \mathcal{M}_h^C)$, we observe that $\bm{v}_h =\bm{u}_h - \phi_p \bm{e}_1  \in \bm{\cK}_h$ and similarly, $\bm{v}_h =\bm{u}_h \pm \phi_p \bm{e}_2  \in \bm{\cK}_h$. Finally, we prove equation \eqref{char2} using \eqref{eq:DVI}.
%\end{proof}
\begin{remark}
Using Lemma \ref{2.3}, we have
	\begin{align}
a(\bm{u}_h,\bm{v}_h)- L(\bm{v}_h)=0, \quad \forall~ \bm{v}_h\in \bm{Z}_h.
\end{align}
\end{remark}
\noindent
In the next two lemmas, we state the standard trace inequality \cite{brenner2007mathematical} and inverse inequalities on discrete space $\bm{V}_h$.
\begin{lemma} \label{Lemmmma}
	Let $ e\in \partial T$ for some $T \in \cT_h$ and $p \in [1, \infty)$. The following holds for any $\bm{\psi} \in \bm{W}^{1,p}(T)$
	\begin{align*}
	\|\bm{\psi}\|^p_{\bm{L}^p(e)} \lesssim h_e^{-1} \Big(h_e^p \|\bm{\nabla} \bm{\psi}\|^p_{\bm{L}^p(T)} + \|\bm{\psi}\|^p_{\bm{L}^p(T)} \Big).
	\end{align*}
\end{lemma}
\noindent
\begin{lemma} \label{Lemama}
	Let $1 \leq p,q \leq \infty$ and $\bm{w}_h \in \bm{V}_h$. Then, it holds that
	\begin{enumerate}
		\item $	\|\bm{w}_h\|_{\bm{W}^{m,p}(T)} \lesssim h^{l-m}_T h_T^{2(\frac{1}{p}-\frac{1}{q})} \|\bm{w}_h\|_{\bm{W}^{l,q}(T)} \quad \forall~T~ \in \cT_h$, $l \leq m$,
		\item $	\|\bm{w}_h\|_{\bm{L}^{\infty}(T)} \lesssim h^{-1}_T \|\bm{w}_h\|_{\bm{L}^2(T)} \quad \forall~ T \in \cT_h$,
		\item $	\|\bm{w}_h\|_{\bm{L}^{\infty}(e)} \lesssim h^{-\frac{1}{2}}_e \|\bm{w}_h\|_{\bm{L}^2(e)} \quad \forall~ e \in \cE_h$.
	\end{enumerate}	
\end{lemma}
\section{Discrete contact force density} \label{sec4}
\noindent
In this section, we begin by introducing the discrete counterpart of $\bm{\lambda}$, which is required in proving the main results. Let $\{y^c_i\}_{i=0}^n \in \cN^C_h$ be the enumeration of vertices on the contact boundary and we denote $\cT^C_h$ to be the mesh formed by the trace of $\cT_h$ on $\Gamma_C$ which is characterized by the subdivision of $\{y^c_i\}_{i=0}^n$. Let us denote an element on $\Gamma_C$ with midpoint $m^c_i$ as $t^c_i=[y^c_i,y^c_{i+1}]$. Therefore, we have the following classification of $\Gamma_C$
\begin{align*}
\Gamma_C =\underset{0\leq i \leq n-1}{\bigcup}  q_i^1 \cup q_i^2,
\end{align*}
where $q_i^1~=~[{y}_i^c,~{m}_{i}^c]$ and $q_i^2~=~[{m}_i^c,~{y}_{i+1}^c]$. It holds that $t^c_i=q_i^1 \cup q_i^2$ for $0 \leq i \leq n-1$. Define
\begin{align} \label{Q_h}
\bm{W}_h :=\{ \bm{v}\in[C(\overline \Gamma_C)]^2~: \bm{v}|_{q_i^j} \in [\mathbb{P}_1(q_i^j)]^2,~0\leq i \leq n-1,~j~=1,2\}.
\end{align}
Let $\{\psi_z\bm{e}_i~: z\in\mathcal{N}_h^C \cup \mathcal{M}_h^C, ~i=1,2\}$ be the nodal Lagrange basis for $\bm{W}_h$, where
{
\begin{align*}
\begin{split}
\begin{aligned}
\psi_z(p) = \begin{cases} &1 ~~~\text{if}~ p~=~z, \\
&0~~~~\text{if}~p~\neq~z, \end{cases}
\end{aligned}~~\quad p \in \mathcal{N}_h^C \cup \mathcal{M}_h^C.
\end{split}
\end{align*}}
In order to define the discrete counterpart of $\bm{\lambda}$, we introduce the interpolation operator $\Xi_h: \bm{W}_h \rightarrow \bm{Z}^c_h$, defined as
\begin{align} \label{interp}
\Xi_h\bm{v}~=\sum_{p \in \mathcal{N}_h^C \cup \mathcal{M}_h^C}\sum_{i=1}^{2}  v_i(p)\phi_p\bm{e}_i \quad \forall~\bm{v}=(v_1,v_2) \in \bm{W}_h.
\end{align}
Moreover, $\Xi_h$ is one-one and onto map and hence, its inverse $\Xi^{-1}_h: \bm{Z}^c_h \rightarrow \bm{W}_h$ exists and is defined by
\begin{align} \label{deff}
\Xi_h^{-1}\bm{v}~=\sum_{z \in \mathcal{N}_h^C \cup \mathcal{M}_h^C}\sum_{i=1}^{2}  v_i(z)\psi_z\bm{e}_i \quad \forall~\bm{v}=(v_1,v_2) \in \bm{Z}_h^c.
\end{align}
The following property of $\Xi^{-1}_h$ is clear from the definition \eqref{deff}.
\begin{align}\label{prop}
\Xi_h^{-1}\bm{v}(z)~= \bm{v}(z)~\forall~~{z \in \mathcal{N}_h^C \cup \mathcal{M}_h^C,}~~\forall~\bm{v} \in \bm{Z}_h^c.
\end{align} 
Now, we define the discrete contact force density $\bm{\lambda}_h \in \bm{W}_h$ as
\begin{align}\label{2.8}
\langle \bm{\lambda}_h, \bm{v}_h \rangle_h = L(\Xi_h\bm{v}_h) - a(\bm{u}_h, \Xi_h\bm{v}_h)~\forall~ \bm{v}_h \in \bm{W}_h,
\end{align}
where
\begin{align} \label{innerp}
\langle \bm{w}_h, \bm{v}_h \rangle_h~= \sum_{p \in \mathcal{N}_h^C \cup \mathcal{M}_h^C} \bm{w}_h(p) \cdot \bm{v}_h(p) \int_{\Gamma_{p,C}} \psi_p~ds.
\end{align} Note that, $\langle \cdot, \cdot \rangle_h$ defines an inner product on $\bm{W}_h \times \bm{W}_h$. We collect the properties of $\bm{\lambda_h}$ in the form of next lemma.
\begin{lemma}\label{lem:disL} For $\bm{\lambda}_h=(\lambda_{h,1},\lambda_{h,2}) \in \bm{W}_h$, it holds that
	\begin{align*}
	\begin{split}
	\begin{aligned}
\lambda_{h,1}(p) ~&\geq 0 \\
\lambda_{h,2}(p)~ &=~ 0
	\end{aligned} \qquad \forall~ p \in \mathcal{N}_h^C \cup \mathcal{M}_h^C.
	\end{split}
	\end{align*}
\end{lemma}
\begin{proof}
	Let the test function $\bm{v}_h \in \bm{W}_h$ be such that 
		\begin{align} \label{trial}
	\begin{split}
	\begin{aligned}
	\bm{v}_h(z) = \begin{cases} &(1,0) ~~~~\text{ if}~z=p, \\
	&(0,0)~~\text{ if}~z \neq p, \end{cases}
	\end{aligned} \quad \forall~ z \in \mathcal{N}_h^C \cup \mathcal{M}_h^C,
	\end{split}
	\end{align}
	where $p \in \mathcal{N}_h^C \cup \mathcal{M}_h^C$ be an arbitrary node. With this choice of $\bm{v_h}$, we have 
	\begin{align} \label{3.4}
	\Xi_h\bm{v}_h =\sum_{z \in \mathcal{N}_h^C \cup \mathcal{M}_h^C}(v_{h,1}(z)\phi_z,v_{h,2}(z)\phi_z)=(\phi_p,0)
	=\phi_p\bm{e}_1.
	\end{align}
Utilizing equations \eqref{char2} and \eqref{3.4}, the following holds
	\begin{align}\label{2.9}
	\langle \bm{\lambda}_h, \bm{v}_h \rangle_h
	&= L(\phi_p\bm{e}_1) - a(\bm{u}_h, \phi_p\bm{e}_1) \geq 0,
	\end{align}
	furthermore, using the equations \eqref{innerp} and \eqref{trial}, we end up on
	\begin{align}\label{2.10}
	\langle \bm{\lambda}_h, \bm{v}_h \rangle_h &= \sum_{z \in \mathcal{N}_h^C \cup \mathcal{M}_h^C} \bm{\lambda}_h(z) \cdot \bm{v}_h(z) \int_{\Gamma_{z,C}} \psi_z~ds \nonumber\\
	&= \bm{\lambda}_h(p) \cdot \bm{v}_h(p) \int_{\Gamma_{p,C}} \psi_p~ds= \lambda_{h,1}(p)\int_{\Gamma_{p,C}} \psi_p~ds.
	\end{align}
	Since $\int_{\Gamma_{p,C}} \psi_p~ds \geq 0$,  using equations $\eqref{2.9}$ and $\eqref{2.10}$, we conclude $\lambda_{h,1}(p)\geq 0$. Analogously, for any $p \in \mathcal{N}_h^C \cup \mathcal{M}_h^C$, we have $\Xi_h\bm{v}_h = \phi_p\bm{e}_2$, where the test function $\bm{v}_h \in \bm{W}_h$ is such that
	\begin{align*}
	\begin{split}
	\begin{aligned}
	\bm{v}_h(z) = \begin{cases} &(0,1) ~~~\text{if}~ z=p, \\
	&(0,0)~~\text{if}~z \neq p, \end{cases}
	\end{aligned} \quad \forall~z \in \mathcal{N}_h^C \cup \mathcal{M}_h^C.
	\end{split}
	\end{align*}
Employing \eqref{char2} and \eqref{trial}, we get
	\begin{align}\label{3.111}
	\langle \bm{\lambda}_h, \bm{v}_h \rangle_h  &= L(\Xi_h\bm{v}_h) - a(\bm{u}_h, \Xi_h\bm{v}_h) \nonumber \\
	&= L(\phi_p\bm{e}_2) - a(\bm{u}_h, \phi_p\bm{e}_2)=0,
	\end{align}
	and
	\begin{align}\label{2.12}
	\langle \bm{\lambda}_h, \bm{v}_h \rangle_h &= \sum_{z \in \mathcal{N}_h^C \cup \mathcal{M}_h^C} \bm{\lambda}_h(z) \cdot \bm{v}_h(z) \int_{\Gamma_{z,C}} \psi_z~ds \nonumber\\
	&= \bm{\lambda}_h(p) \cdot \bm{v}_h(p) \int_{\Gamma_{p,C}} \psi_p~ds = \lambda_{h,2}(p)\int_{\Gamma_{p,C}} \psi_p~ds.
	\end{align} 
	With the help of \eqref{3.111} and {noting} $ \int_{\Gamma_{p,C}} \psi_p~ds > 0$ in \eqref{2.12}, {we have} $\lambda_{h,2}(p)=0$. Since $p \in \mathcal{N}_h^C \cup \mathcal{M}_h^C$ was an arbitrary node, we have the desired result of this lemma.
\end{proof}

\section{quasi discrete contact force density} \label{sec5}
\noindent
In this section, we introduce the quasi discrete contact force density $\bm{\bar{\lambda}}_h \in \bm{V}^*$ which will play an important role in proving the reliability of a posteriori error estimator. First, we collect some tools in order to define $\bm{\bar{\lambda}}_h$. We recall the linear residual $\bm{\mathcal{L}}_h \in \bm{V}^*$ as follows
\begin{align*}
\langle \bm{\mathcal{L}}_h , \bm{v} \rangle_{-1,1} &:= L(\bm{v}) -a(\bm{u}_h,\bm{v}) \quad  \forall ~\bm{v} \in \bm{V}.
\end{align*} 
\begin{remark} \label{4.1}
Let $\bm{v}=(v_1,v_2) \in \bm{V}$, then we write
\begin{align*}
\langle \bm{\mathcal{L}}_h, \bm{v} \rangle_{-1,1} = \sum_{i=1}^2 \langle \mathcal{L}_{h,i}, {v_i} \rangle_{-1,1},
\end{align*}
where
{
\begin{align*}
\langle \mathcal{L}_{h,i}, {v_i} \rangle_{-1,1} &:= L(\bm{w}_i) -a (\bm{u}_h, \bm{w}_i), \quad i=1,2,
\end{align*}}
\noindent
with $\bm{w}_1=(v_1,0)$ and $\bm{w}_2=(0,v_2)$.
\end{remark}
\noindent
Since $\bm{V}_h \subseteq \bm{V}$, we define $\bm{\mathcal{L}}_h \in \bm{V}_h$ as
\begin{align} \label{Rlin}
\langle  \bm{\mathcal{L}}_h , \bm{v}_h \rangle_{-1,1} &:= L(\bm{v}_h) -a(\bm{u}_h,\bm{v}_h), \quad  \forall~ \bm{v}_h \in \bm{V}_h.
\end{align} 
In the next lemma, the key relation between $ \bm{\mathcal{L}}_h$ and $\bm{\lambda}_h$ is {obtained}.
\begin{lemma} \label{relation}
	For $i=1,2$ and $p \in \cN^C_h \cup \cM^C_h$, it holds that 
	\begin{align}
		\langle \bm{\lambda}_h, \psi_p \bm{e}_i \rangle_h = \langle  \bm{\mathcal{L}}_h , \phi_p \bm{e}_i \rangle_{-1,1}.
	\end{align}
\end{lemma}
\begin{proof}
For $i=1,2$, we substitute $\bm{v}_h=\phi_p \bm{e}_i$, $p \in \cN^C_h \cup \cM^C_h$ in \eqref{Rlin} to obtain 
\begin{align}
\langle  \bm{\mathcal{L}}_h ,\phi_p \bm{e}_i \rangle_{-1,1} = L(\phi_p \bm{e}_i) -a(\bm{u}_h,\phi_p \bm{e}_i).
\end{align} 
Using \eqref{interp}, we have 
\begin{align}
\langle  \bm{\mathcal{L}}_h ,\phi_p \bm{e}_i \rangle_{-1,1} = L(\Xi_h(\psi_p \bm{e}_i)) -a(\bm{u}_h,\Xi_h(\psi_p \bm{e}_i)).
\end{align} Hence, we have the desired {result} taking into account \eqref{2.8}.
\end{proof}
\noindent
For the sake of convenience, we further use  the following notations to describe the jump terms across different part of boundaries.  Let $e$ be an interior edge such that $e \in \partial T \cap \partial \tilde{T}$ and $\bm{n}$ be the outward unit normal to $T$, then we define
\begin{align} \label{Ijump}
\bm{\cJ}_e^I(\bm{u}_h):= (\bm{\sigma}|_{\tilde{T}}(\bm{u}_h)-\bm{\sigma}|_{T}(\bm{u}_h))\bm{n},
\end{align}
where $\bm{\cJ}_e^I(\bm{u}_h)=
(\cJ_{1,e}^I(\bm{u}_h),\cJ_{2,e}^I(\bm{u}_h))'$.   We have
\begin{align*}
\sum_{p \in \mathcal{N}_h \cup \mathcal{M}_h} \int_{\Gamma_{p,I}} \bm{\cJ}^I(\bm{u}_h)~ds = \sum_{p \in \mathcal{N}_h \cup \mathcal{M}_h} \sum_{e \in \Gamma_{p,I}}\int_{e} \bm{\cJ}_e^I(\bm{u}_h)~ds,
\end{align*}
where $\bm{\cJ}^I(\bm{u}_h)|_e=\bm{\cJ}_e^I(\bm{u}_h)$. Further, let $e \in \Gamma_{p,N}$ be the edge on  Neumann boundary, then we define the jump term $\bm{\cJ}_e^N(\bm{u}_h)=
(\cJ_{1,e}^N(\bm{u}_h),\cJ_{2,e}^N(\bm{u}_h))'$, as follows
\begin{align*}
\bm{\cJ}_e^N(\bm{u}_h) := \bm{g}-\hat{\bm{\sigma}}|_{T}(\bm{u}_h),
\end{align*}
where $e \in \partial T$ and
\begin{align*}
\sum_{p \in \mathcal{N}_h^N \cup \mathcal{M}_h^N} \int_{\Gamma_{p,N}} \bm{\cJ}^N(\bm{u}_h)~ds:= \sum_{p \in \mathcal{N}_h^N \cup \mathcal{M}_h^N} \sum_{e \in \Gamma_{p,N}}\int_{e} \bm{\cJ}_e^N(\bm{u}_h)~ds,
\end{align*} 
with $\bm{\cJ}^N(\bm{u}_h)|_{e}= \bm{\cJ}_e^N(\bm{u}_h).$
Employing integration by parts formula and \eqref{eq:ES}, we have the following representation for $\langle \bm{\mathcal{L}}_h , \bm{v}_h \rangle_{-1,1}$
\begin{align} \label{residual}
\langle  \bm{\mathcal{L}}_h , \bm{v}_h \rangle_{-1,1} &=\sum_{i=1}^{2}  \sum_{p \in \mathcal{N}_h \cup \mathcal{M}_h} \Big[L(v_{h,i}(p) \phi_p \bm{e}_i) - a(\bm{u}_h,v_{h,i}(p) \phi_p \bm{e}_i ) \Big] \nonumber \\ & = \sum_{i=1}^{2}  \sum_{p \in \mathcal{N}_h \cup \mathcal{M}_h} \int_{\Omega_p} \bm{s}(\bm{u}_h)  \cdot v_{h,i}(p) \phi_p \bm{e}_i~dx \nonumber \\ & \hspace{0.3cm}- \sum_{i=1}^{2}  \sum_{p \in \mathcal{N}_h \cup \mathcal{M}_h} \int_{\Gamma_{p,I}} \bm{\cJ}_e^I(\bm{u}_h) \cdot v_{h,i}(p) \phi_p \bm{e}_i~ds \nonumber \\ & \hspace{0.3cm}+ \sum_{i=1}^{2}  \sum_{p \in \mathcal{N}_h^N \cup \mathcal{M}_h^N} \int_{\Gamma_{p,N}} \bm{\cJ}_e^N(\bm{u}_h)\cdot v_{h,i}(p) \phi_p \bm{e}_i~ds \nonumber\\ & \hspace{0.3cm}- \sum_{i=1}^{2}  \sum_{p \in \mathcal{N}_h^C \cup \mathcal{M}_h^C} \int_{\Gamma_{p,C}} \bm{\sigma}(\bm{u}_h) \bm{n} \cdot v_{h,i}(p) \phi_p \bm{e}_i~ds,
\end{align}
where $\bm{s}(\bm{u}_h):=(s_1(\bm{u}_h),s_2(\bm{u}_h))'=\bm{f} + \bm{\text{div}\sigma}(\bm{u}_h)$. For $i=1,2$, we insert the Lemma \ref{2.3} in \eqref{residual} to derive
\begin{align}
\langle  \bm{\mathcal{L}}_h ,\phi_p \bm{e}_i \rangle_{-1,1} = 0 \quad \forall~p \in  (\mathcal{N}_h \cup \mathcal{M}_h) \setminus (\mathcal{N}_h^C \cup \mathcal{M}_h^C), \label{prop1}
\end{align}and
\begin{align}
\langle  \bm{\mathcal{L}}_h ,\phi_p \bm{e}_2 \rangle_{-1,1}=0 \quad \forall~p \in  \mathcal{N}_h^C \cup \mathcal{M}_h^C.\label{prop2}
\end{align}
Motivated from the articles \cite{fierro2003posteriori}  and  \cite{nochetto2005fully,krause2015efficient}, we define the quantity called quasi-discrete contact force density differently to the distinct parts of $\Gamma_C$ in order to achieve the local efficiency estimates.  The article \cite{moon2007posteriori} discussed the mentioned idea in detail to introduce the discrete version of the contact force density in dealing with the parabolic variational inequalities.
First,  we divide actual contact nodes, i.e., $u_{h,1}(p)=\chi_h(p)$ where $p \in \cN^C_h \cup \cM^C_h$ into two distinct categories. The set of full contact nodes is denoted by $\mathcal{N}^{FC}_h:=\{p \in \mathcal{N}_h^C \cup \mathcal{M}_h^C~|~u_{h,1}= \chi_h~\text{on}~\Gamma_{p,C}\}$ and the remaining actual contact nodes are called semi contact nodes and denoted by $\cN^{SC}_h$.
We set $\cN_h^{NC}$ for no actual contact nodes i.e., for $p \in \cN^{NC}_h$, $u_{h,1}(p) \neq \chi_h(p)$. Let $p \in \cN^C_h \cup \cM^C_h$ and define $\bm{s}_p:=(s_{p,1},s_{p,2})$ where 
\begin{align} \label{spd}
s_{p,1}:= \frac{ \langle  \bm{\mathcal{L}}_h , \phi_p \bm{e}_1 \rangle_{-1,1}}{\int_{\Gamma_{p,C}} \psi_p~ds}= \frac{	\langle \bm{\lambda}_h, \psi_p \bm{e}_1 \rangle_h}{\int_{\Gamma_{p,C}} \psi_p~ ds}~~\text{and}~~s_{p,2}:=\lambda_{h,2}(p)=0.
\end{align} 
Let $\bar{\Gamma}_{p,C} \subsetneq \Gamma_{p,C}$ be an contact edge corresponding to node $p$, then for full-contact node $p \in \cN^{FC}_h$ and semi-contact node $p \in \cN^{SC}_h$, we define the scalars
\begin{align} \label{eq:DPD}
e_p(v_1):= \frac{\int_{\bar{\Gamma}_{p,C}} v_1 \psi_p~ds}{\int_{\bar{\Gamma}_{p,C}} \psi_p~ds}.
\end{align}
For $ p \in \cN^{NC}_h$, the scalars $e_p(v_1)$ are defined similarly as in \eqref{eq:DPD}. For all $p \in \mathcal{N}_h^0 \cup \mathcal{M}_h^0$ except the contact nodes with $i= 1$, we set 
\begin{align} \label{nonc}
e_p(v_i):= \frac{\int_{\Omega_p} v_i \phi_p~ds}{\int_{\Omega_p} \phi_p~ds}.
\end{align}
{ Lastly,  for nodes on the Dirichlet boundary we set $e_p(v_i)=0.$ }
 {These} choices for $\bm{s}_p$ and $e_p(v_i)$ are important for the analysis in the next section. The next lemma provide approximation properties corresponding to the constants $e_p(v_i)$. It can be proved using the similar ideas mentioned in \cite{walloth2012,verfurth1996review}. 
	\begin{lemma} \label{approxprop}
	Let $p \in \cN^C_h \cup \cM^C_h$ and $w \in W^{1,1}(\Omega_p)$. {For $e_p(w)= \frac{\int_{\Gamma^*_{p,C}} w \psi_p~ds}{\int_{\Gamma^*_{p,C}} \psi_p~ds}$ defined in \eqref{eq:DPD}}, where $\Gamma^*_{p,C} \subseteq \Gamma_{p,C}$ (e.g. $\bar{\Gamma}_{p,C}$ or $\Gamma_{p,C}$), the following {estimates} hold
	\begin{align}
	\|w- e_p(w)\|_{L^1(\Omega_p)} &\lesssim h_p \|\nabla w\|_{L^1(\Omega_p)}, \label{approx1}\\  \|\nabla(w- e_p(w)) \|_{L^1(\Omega_p)} &\lesssim  \|\nabla w\|_{L^1(\Omega_p)}. \label{approx2}
	\end{align}
\end{lemma}
\begin{lemma}\label{approxprop1}
	Let $T \in \cT_h$, $\phi \in L^1(T)$ and $e_p(\phi)$ be the scalars defined in \eqref{nonc}. Then, the following hold
	\begin{align*}
	\|\phi - e_p(\phi)\|_{L^1(\Omega_p)} &\lesssim h_p\|\nabla \phi\|_{L^1(\Omega_p)}, \\
	\|\phi - e_p(\phi)\|_{L^1(e)} & \lesssim h_p^{1/2}\|\nabla \phi\|_{L^1(\Omega_p)},
	\end{align*}
	where $p \in \cN^T_h \cup \cM^T_h$ and $e \subset \partial T.$
\end{lemma}
{
\begin{remark}
Note that, by taking into account Poincar$\acute{e}$-Fredrichs inequality and trace inequality,  the estimates in Lemma \ref{approxprop1} are also valid for constants $e_p(\phi),~p \in \cN^D_h \cup \cM^D_h $.
\end{remark}
\noindent
Exploiting the definition of constants $e_p(\cdot)$ and property of nodal basis function $\underset {p \in \mathcal{N}_h \cup \mathcal{M}_h}{\sum} \phi_p =1,$ the quasi discrete contact force density $\bm{\bar{\lambda}_{h}} \in \bm{V^*}$ is defined in the following way
\begin{align} \label{eq:quasid}
\langle \bm{\bar{\lambda}}_h, \bm{v} \rangle_{-1,1} := \sum_{i=1}^{2} \langle \bar{\lambda}_{h,i} , v_i \rangle_{-1,1} \quad \forall~ \bm{v}=(v_1,v_2) \in \bm{V},
\end{align}
where for $i=1,2$
\begin{align} \label{def}
\langle \bar{\lambda}_{h,i} , v_i \rangle_{-1,1}&=
\sum_{p \in \mathcal{N}_h \cup \mathcal{M}_h}\langle \bar{\lambda}_{h,i},  v_i\phi_p \rangle_{-1,1} :=
\sum_{p \in \mathcal{N}_h \cup \mathcal{M}_h} \langle  \bm{\mathcal{L}}_h,  \phi_p \bm{e_i} \rangle_{{-1,1}} e_p(v_i). 
\end{align}}
\begin{remark} \label{sign}
We note that for any sufficiently small $\kappa >0$, we have $\bm{v}_h = \bm{u}_h+ \kappa \phi_p\bm{e}_1 \in \boldsymbol{\mathcal{K}}_h~\forall~p \in \cN^{NC}_h$. Using Lemma \ref{char2} and Lemma \ref{relation}, we deduce $\langle \bm{\lambda}_h, \psi_p \bm{e}_1 \rangle_h =0 \quad \forall~p \in \mathcal{N}_h^{NC}$.
\end{remark}
\noindent
The following two lemmas are the consequences of equations \eqref{eq:quasid}-\eqref{def}.
\begin{lemma} \label{sign1}
	It holds that
	\begin{align}
	\langle \bar{\lambda}_{h,1} , v_1 \rangle_{-1,1} = \sum_{p \in \mathcal{N}_h^C \cup \mathcal{M}_h^C} \langle \bar{\lambda}_{h,1}, v_1 \phi_p \rangle_{-1,1}.
	\end{align}
\end{lemma}
\begin{proof}
{
Utilizing the definition of quasi discrete contact force density $\bar{\bm{\lambda}}_{h}$,  we find 
\begin{align}\label{1311}
\langle \bar{\lambda}_{h,1}, v_1 \rangle_{-1,1} = \sum_{p \in \mathcal{N}_h \cup \mathcal{M}_h}\langle \bar{\lambda}_{h,1}, v_1\phi_p \rangle_{-1,1} =\sum_{p \in \mathcal{N}_h \cup \mathcal{M}_h} \langle  \bm{\mathcal{L}}_h,  \phi_p \bm{e_1} \rangle_{{-1,1}} e_p(v_1) .
\end{align}
Further,  taking into account that $\langle \bm{\mathcal{L}}_h,  \phi_p \bm{e_1} \rangle_{{-1,1}}=0~\forall~p \in (\mathcal{N}_h \cup \mathcal{M}_h) \backslash (\mathcal{N}^C_h \cup \mathcal{M}^C_h)$,  the last equation \eqref{1311} reduces to

\begin{align} \label{131}
\langle \bar{\lambda}_{h,1}, v_1 \rangle_{-1,1} = \sum_{p \in \mathcal{N}_h^C \cup \mathcal{M}_h^C} \langle  \bm{\mathcal{L}}_h,  \phi_p \bm{e_1} \rangle_{{-1,1}} e_p(v_1) .
\end{align}}
Thus,  we have 
	\begin{align}\label{t1}
	\langle \bar{\lambda}_{h,1}, v_1 \rangle_{-1,1} = \sum_{p \in \mathcal{N}_h^C \cup \mathcal{M}_h^C} \langle \bar{\lambda}_{h,1}, v_1 \phi_p \rangle_{-1,1}.
	\end{align}
	This completes the proof. 
\end{proof}
\begin{lemma} \label{sign2}
	It holds that
	\begin{align}
	\langle  \bm{\bar{\lambda}}_h, \bm{v} \rangle_{-1,1} = \langle  \bar{\lambda}_{h,1}, v_1 \rangle_{-1,1} \geq 0 \quad \forall~ v_1 \geq 0,
	\end{align}
	where $\bm{v}=(v_1,v_2) \in \bm{V}.$
	\begin{proof}
{
In view of equations \eqref{prop1} and \eqref{prop2},  we find 
$\langle  \bm{\mathcal{L}}_h,  \phi_p \bm{e_2} \rangle_{{-1,1}} =0~\forall~p \in \mathcal{N}_h \cup \mathcal{M}_h$. Thus,  using relation \eqref{def} we have 
\begin{align*}
\langle  \bm{\bar{\lambda}}_h, \bm{v} \rangle_{-1,1} = \langle  \bar{\lambda}_{h,1}, v_1 \rangle_{-1,1}. \end{align*} 
Further with the help of equation \eqref{t1},  we find
\begin{align*}
\langle  \bm{\bar{\lambda}}_h, \bm{v} \rangle_{-1,1} = \langle  \bar{\lambda}_{h,1}, v_1 \rangle_{-1,1} =  \sum_{p \in \mathcal{N}_h^C \cup \mathcal{M}_h^C} \langle \bar{\lambda}_{h,1}, v_1 \phi_p \rangle_{-1,1}.
\end{align*}
Exploiting relation \eqref{relation} and the definition of constant $\bm{s_p}$, we find
 \begin{align}\label{rel10}
\bm{\langle \bm{\bar{\lambda}_{h}}} , \bm{v} \rangle_{{-1,1}}&= \sum_{p \in \mathcal{N}_h^C \cup \mathcal{M}_h^C}\langle \bm{\mathcal{L}_h},  \phi_p \bm{e_1} \rangle_{{-1,1}} e_p(v_1) \nonumber\\
&= \sum_{p \in \mathcal{N}_h^C \cup \mathcal{M}_h^C}\langle \bm{\lambda_{h}}, \psi_p \bm{e_1} \rangle_{{h}}e_p(v_1) \nonumber\\
&=\sum_{p \in \mathcal{N}_h^C \cup \mathcal{M}_h^C}s_{p,1} e_p(v_1) \int_{\gamma_{p,C}}\psi_p~ds.
\end{align}
Using Lemma \ref{lem:disL}, we deduce that $s_{p,1} = \lambda_{h,1}(p) \geq 0$.  Hence the assertion follows by taking into account that  $e_p(v_1) \geq 0 ~\forall~v_1 \geq 0$ together with the fact that $\int_{\gamma_{p,C}}\psi_p~ds>0$.}
	\end{proof}
\end{lemma}
\par
\noindent
The sign property of quasi discrete contact force {discussed in} Lemma \ref{sign2} plays a key role in proving the reliability estimates.
\section{A posteriori estimates} \label{sec6}
\noindent
In this section, we begin by defining the estimators
	\begin{align*}
\eta_1&= \underset{p \in \mathcal{N}_h \cup \mathcal{M}_h}{max}~\eta_{1,p}, \quad\quad\text{with} \qquad~~~\eta_{1,p}:=h_p^2 \|\bm{s}(\bm{u}_h)\|_{\bm{L}^{\infty}(\Omega_p)},\\
\eta_2&= \underset{p \in \mathcal{N}_h \cup \mathcal{M}_h}{max}~\eta_{2,p},\quad \quad\text{with} \qquad~~~\eta_{2,p}:=h_p \|\bm{\cJ}_e^I(\bm{u}_h)\|_{\bm{L}^{\infty}(\Gamma_{p,I})} ,\\
\eta_3&= \underset{p \in \mathcal{N}_h^N \cup \mathcal{M}_h^N}{max}~\eta_{3,p}, \quad \quad~\text{with} \qquad~~~\eta_{3,p}:=h_p \|\bm{\cJ}_e^N(\bm{u}_h)\|_{\bm{L}^{\infty}(\Gamma_{p,N})},\\
\eta_4&= \underset{p \in \mathcal{N}_h^C \cup \mathcal{M}_h^C}{max}~\eta_{4,p}, \qquad\text{with} \qquad~~~\eta_{4,p}:=h_p { \|\hat{\sigma_2}(\bm{u}_h)\|_{\bm{L}^{\infty}(\Gamma_{p,C})}},\\
\eta_5&= \underset{p \in \mathcal{N}_h^C \cup \mathcal{M}_h^C}{max}~\eta_{5,p}, \qquad\text{with} \qquad~~~\eta_{5,p}:=h_p \|\hat{\sigma_1}(\bm{u}_h)\|_{\bm{L}^{\infty}(\Gamma_{p,C})}.\\
\end{align*}
\subsection{Reliability of the error estimator}
Define the space $$\bm{U}:= \{\bm{v} \in \bm{W}^{2,1}(\Omega)~;~ \bm{\sigma(\bm{v})\bm{n}}=0~ \text{on}~ \Gamma_C \cup \Gamma_N\}.$$ We now state the main result of this section, namely the reliability of the error estimator $\eta_h$ {which is} defined in equation \eqref{est}.
\begin{theorem} \label{thm:rel}
	Let $\bm{u}$ and $\bm{u}_h$ be the continuous and discrete solution of equations \eqref{eq:CVI} and \eqref{eq:DVI}, respectively. Let $\bm{\lambda}$ and $\bar{\bm{\lambda}}_h$ be defined as in \eqref{eq:sigmadef} and \eqref{eq:quasid}, respectively. Then, the following error bound holds:
	 \begin{equation}  \label{rel1}
	\max{\{\|\bm{u}-\bm{u}_h\|_{\bm{L}^{\infty}(\Omega)}~,~ \|\bm{\lambda}-\bar{\bm{\lambda}}_h\|_{-2,\infty,\Omega}\}} \lesssim \eta_h,
	\end{equation}
where the operator norm $\|\cdot\|_{-2,\infty,\Omega}$ is defined by
	\begin{align} \label{norm}
	\|\bm{q}\|_{-2,\infty,\Omega} := \text{sup} \{ \langle \bm{q}, \bm{v} \rangle_{-1,1} : \bm{v} \in \bm{V} \cap \bm{U}~,~ \|\bm{v}\|_{\bm{W}^{2,1}(\Omega)}\leq 1\},
	\end{align} and the error estimator $\eta_h$ is given by
	\begin{align} \label{est}
 \eta_h&:={l_h} \Psi + \|(u_{h,1}-\chi)^{+}\|_{L^{\infty}(\Gamma_C)} + \|(\chi-u_{h,1})^+\|_{L^{\infty}(\Lambda_h^{C})},
	\end{align} 
with
	\begin{align*}
{l_h=1+|\text{log}(h_{min})|^2} \quad;\quad	\Psi:= \sum_{i=1}^{5} \eta_i\quad;~~	{\Lambda}^C_h:= \{\Gamma_{p,C}:p \in \cN^C_h \cup \cM^C_h~\text{such that}~ \langle \bm{\lambda}_h, \psi_p \bm{e}_1 \rangle_h >0 \}.
	\end{align*} 
\end{theorem}
\par
\noindent
{To prove the estimate \eqref{rel1}, we begin by introducing} the residual functional $\bm{\mathcal{G}}_h:\bm{V} \rightarrow \mathbb{R}$ as
\begin{align}\label{eq:GAL}
\bm{\mathcal{G}}_h(\bm{v}) := a(\bm{u}-\bm{u}_h,~\bm{v}) + \langle \bm{\lambda}-\bm{\bar{\lambda}}_h,\bm{v} \rangle_{-1,1} \quad \forall~ \bm{v}~\in{\bm{V}}.
\end{align}
In the later analysis, we need an extension of the bilinear form $a(\cdot,\cdot)$ which would allow to test with functions less regular than $\bm{H^1(\Omega)}$.  For $p >2, 1 \leq q < 2$ and $\frac{1}{p}+\frac{1}{q}=1$, let \begin{align*}
	\bm{V}_0:= \bm{V} + \bm{W}^{1,p}(\Omega), \quad \bm{V}_1:= \bm{V}+ \bm{W}^{1,q}(\Omega).
	\end{align*}
	\par
	\noindent
	Let $\tilde{a}(\bm{z},\bm{v})$ denotes the extended bilinear form on $\bm{V}_0 \times \bm{V}_1$ defined by $\tilde{a}(\bm{z},\bm{v}):= \int_{\Omega}\bm{\sigma(z)} : \bm{\epsilon(v)}~dx$ which is such that $\tilde{a}(\bm{z},\bm{v})={a}(\bm{z},\bm{v}) ~~\forall~ \bm{z},\bm{v} \in \bm{V}$.
	\noindent
	Then, $\bm{\lambda}^* \in \bm{V}_1^*$ is defined by
	\begin{align}\label{eq:extsigmadef}
	\langle \bm{\lambda}^*, \bm{v}\rangle=L(\bm{v}) -\tilde{a}(\bm{u},\bm{v})\quad \forall~\bm{v}\in \bm{V}_1.
	\end{align} For any $\bm{v} \in \bm{V}_1$, we define $\tilde{\bm{\mathcal{G}}_h}$ by
	\begin{align} \label{eq:extGAL}
	\tilde{\bm{\mathcal{G}}_h}(\bm{v}):= \tilde{a}(\bm{u}-\bm{u}_h,\bm{v}) + \langle \bm{\lambda}^*-\bar{\bm{\lambda}}_h, \bm{v} \rangle_{-1,1}.
	\end{align}
	Lastly, it holds that $\bm{\mathcal{G}_h}(\bm{v}) =\tilde{\bm{\mathcal{G}}_h}(\bm{v}) \quad \forall~ \bm{v} \in \bm{V}$.
	\begin{remark}
		We mention the extended notations for the linear residual as
			\begin{align} \label{extres}
		\langle \tilde{ \bm{\mathcal{L}}}_h , \bm{v} \rangle_{-1,1} &:= L(\bm{v}) -\tilde{a}(\bm{u}_h,\bm{v}), \quad  \forall~ \bm{v} \in \bm{V}_1,
		\end{align} 
		with components defined as
		\begin{align*}
		\langle \tilde{ \bm{\mathcal{L}}}_{h}, \bm{v} \rangle_{-1,1} = \sum_{i=1}^2 \langle  \tilde{\mathcal{L}}_{h,i}, {v_i} \rangle_{-1,1},
		\end{align*}
		where $ \tilde{\mathcal{L}}_{h,i}$ are chosen in the way as $\mathcal{L}_{h,i}$ in Remark \eqref{4.1} for $i=1,2$.
	\end{remark}
\begin{remark} Using equation \eqref{prop1} and \eqref{prop2}, we have
	\begin{align}
	e_p(v_i)\langle \tilde{ \bm{\mathcal{L}}}_{h},\phi_p \bm{e}_i \rangle_{-1,1}  &= 0 \quad \forall~p \in  (\mathcal{N}_h  \cup \mathcal{M}_h) \setminus (\mathcal{N}_h^C \cup \mathcal{M}_h^C),~i=1,2, \label{prop11}\\e_p(v_2) \langle \tilde{ \bm{\mathcal{L}}}_{h},\phi_p \bm{e}_2 \rangle_{-1,1} &=0 \quad \forall~ p \in  \mathcal{N}_h^C \cup \mathcal{M}_h^C, \label{prop22}
	\end{align}
	where $\bm{v}= (v_1,v_2) \in \bm{V}_1$.
\end{remark}
\noindent
{Next, we} build machinery which will lead us to prove Theorem \ref{thm:rel}. We define the upper and lower barrier function of solution of {the} variational inequality \eqref{eq:CVI} and their construction involves the corrector function $\bm{\zeta}=(\zeta_1,\zeta_2) \in \bm{V}$ which satisfies
	\begin{align} \label{eq:Corrector}
 \int_{\Omega} \bm{\sigma}(\bm{\zeta}): \bm{\epsilon}(\bm{v})dx= \bm{\mathcal{G}}_h( \bm{v}) \quad \forall~\bm{v} \in \bm{V}.
\end{align}
{Infact, $\bm{\zeta}$ is the Riesz representation of $\bm{\mathcal{G}}_h$ \cite{evans2010partial}.} Besides $\bm{\zeta}$, we introduce only computable quantities which accounts for the consistency errors. Let $\bm{d}$ be a $2 \times 1$ function having both components as $\|\bm{\zeta}\|_{\bm{L}^{\infty}(\Omega)}:= \max \Big\{\|\zeta_1\|_{L^{\infty}(\Omega)},\|\zeta_2\|_{L^{\infty}(\Omega)}\Big\}$. Let $\bm{b}$ and $\bm{y}$ be a $2 \times 1$ functions  having both components as $ \|(u_{h,1}-\chi)^{+}\|_{L^{\infty}(\Gamma_C)}$ and $\|(\chi-u_{h,1})^+\|_{L^{\infty}({\Lambda}^C_h)}$, respectively. Next, the upper and lower {barrier functions} of $\bm{u}$ are defined by
\begin{align}
\bm{u^{\wedge}}&=\bm{\zeta}+\bm{u}_h+\bm{d}+\bm{y}, \label{eq:upperb}\\ \bm{u_{\vee}}&=\bm{\zeta}+\bm{u}_h-\bm{d}-\bm{b}. \label{eq:lowerb}
\end{align}
In the next two lemmas, we derive the key properties corresponding to $\bm{u^{\wedge}}$ and $\bm{u_{\vee}}$.
	\begin{lemma} \label{lem:barrier}
Let $\bm{u}$ be a continuous solution satisfying \eqref{eq:CVI} and let $\bm{u^{\wedge}}$ be as defined in \eqref{eq:upperb}. Then 
	\begin{align*}
	\bm{u} \leq\bm{u^{\wedge}}.
	\end{align*}
\end{lemma}
\begin{proof}
		 To prove $\bm{u} \leq \bm{u^{\wedge}}$, it is equivalent to show that $\bm{z}:=(\bm{u}-\bm{u^{\wedge}})^{+}=\bm{0}$ in $\Omega$. From definition of $\bm{u^{\wedge}}$, we observe that
	\begin{align*}
	\bm{u}-\bm{u^{\wedge}} &=\bm{u}- \bm{\zeta}-\bm{u}_h-\bm{d}-\bm{y}\\
	& \leq \bm{u}-\bm{u}_h \leq \bm{0}
	\end{align*}
	on $\Gamma_D.$ Hence, $(\bm{u}- \bm{u^{\wedge}})^{+}=\bm{0}$ on $\Gamma_D$. In view of Poincar$\acute{e}$ inequality, our claim will hold true if we show that $\|\nabla \bm{z}\|_{\bm{L}^2(\Omega)}=0$. Using coercivity of $a(\cdot,\cdot)$, \eqref{eq:Corrector}, \eqref{eq:GAL},  Lemma \ref{Lemmaa}, Lemma \ref{sign2} together with equations \eqref{def} and \eqref{131}, we have
	\begin{align*} 
	\|\nabla \bm{z}\|^2_{\bm{L}^2(\Omega)}&\lesssim a(\bm{z},\bm{z}) = a(\bm{u}-\bm{u^{\wedge}},\bm{z}) = a(\bm{u}-\bm{\zeta}-\bm{u}_h,\bm{z}) \notag\\ &= a(\bm{u}, \bm{z})-a(\bm{u}_h,\bm{z}) - \bm{\mathcal{G}}_h(\bm{z})= \langle  \bm{{\bar{\lambda}}}_h- \bm{\lambda} , \bm{z} \rangle_{-1,1} \\ & \leq \langle  \bar{\lambda}_{h,1}, z_1 \rangle_{-1,1} = \sum_{p \in \mathcal{N}_h^C \cup \mathcal{M}_h^C}  \langle \bm{\mathcal{L}}_h,  \phi_p \bm{e}_1 \rangle_{-1,1}  e_p(z_1)  \\ & =\sum_{p \in \mathcal{N}_h^C \cup \mathcal{M}_h^C}  \langle \bm{\lambda}_h, \psi_p \bm{e}_1 \rangle_h e_p(z_1).
	\end{align*}
%	We have used the equation \eqref{i0i} in the last step. 
	{Now,} it is enough to show there does not exist any node $p \in \cN^C_h \cup \mathcal{M}_h^C $ such that $\langle \bm{\lambda}_h, \psi_p \bm{e}_1 \rangle_h e_p(z_1)  >0$. If $p \in \cN^{NC}_h$, then using Remark \eqref{sign}, we have $ \langle \bm{\lambda}_h, \psi_p \bm{e}_1 \rangle_h =0$. For  $p \in \cN^{FC}_h$,  assume on the contrary that $\langle \bm{\lambda}_h, \psi_p \bm{e}_1 \rangle_h e_p(z_1)  >0$ , then there exists a $x^* \in \bar{\Gamma}_{p,C} \subsetneq \Gamma_{p,C}$ such that $z_1(x^*)>0$. Then,
	\begin{align*} 
	z_1(x^*)&>0 \\ \implies u_1(x^*) &> u^{\wedge}_1(x^*) \geq u_{h,1}(x^*)+\|(\chi-u_{h,1})^+\|_{L^{\infty}(\Lambda^C_h)} \geq  \chi(x^*),
	\end{align*}
	which does not hold as $x^* \in \Gamma_C$. The proof follows on the similar lines if $p \in \cN^{SC}_h$, hence the result of lemma holds.
\end{proof}
\begin{lemma} \label{lem:barrier1}
Let $\bm{u}$ be the solution of continuous variational inequality \eqref{eq:CVI} and let $\bm{u_{\vee}}$ be as defined in \eqref{eq:lowerb}. Then, it holds that
\begin{align*}
 \bm{u_{\vee}} \leq \bm{u}.
\end{align*}
\end{lemma}
\begin{proof}
 The proof uses the same ideas as in Lemma \ref{lem:barrier}. Let $\bm{z}=(\bm{u_{\vee}}-\bm{u})^{+}$ and we show that $\bm{z}=\bm{0}$ in $\Omega$. First, note that $\bm{z}|_{\Gamma_D}=\bm{0}$, as we have
	\begin{align*}
	(\bm{u_{\vee}}-\bm{u})|_{\Gamma_D}&=(\bm{u}_h+(\bm{\zeta}-\bm{d})-\bm{b}-\bm{u})|_{\Gamma_D}\\
	&\leq (\bm{u}_h-\bm{u})|_{\Gamma_D} \leq \bm{0}.
	\end{align*}
	Therefore $\bm{z} \in \bm{V}$, thereby it is sufficient to prove $\|\nabla \bm{z}\|_{\bm{L}^2(\Omega)}=0$. Employing the coercivity of $a(\cdot,\cdot)$, equations \eqref{eq:Corrector},\eqref{eq:GAL} and using Lemma \ref{sign2} together with  \eqref{repre},  we find
	\begin{align*}
	\|\nabla \bm{z}\|^2_{\bm{L}^2(\Omega)}&\lesssim a(\bm{z},\bm{z}) =a(\bm{u_{\vee}}-\bm{u},\bm{z})= a(\bm{u}_h+\bm{\zeta}-\bm{u},\bm{z}) \\ &= a(\bm{u}_h-\bm{u},\bm{z})+ \bm{\mathcal{G}}_h(\bm{z})= \langle \bm{\lambda} - {\bm{\bar{\lambda}}_h} , \bm{z} \rangle_{-1,1}  \\ & \leq \langle \bm{\lambda} ,\bm{z} \rangle_{-1,1} = \langle \lambda_1 ,z_1 \rangle_{-1,1}.
	\end{align*}
%	We inserted  in the last step. 
	Using \eqref{eq:IUY}, it holds that $supp(\lambda_1) \subset \{u_1 =\chi\}$. We prove that, $supp(\lambda_1) ~\cap~ supp(z_1) =  \emptyset$. Let us consider
	\begin{align*} 
	z_1(x)>0 \implies u_{\vee,1}(x) &>u_1(x).
	\end{align*}
	From the definition of $ \bm{u_{\vee}}$, we have
	\begin{align*}
u_1(x) <	u_{\vee,1}(x)= u_{h,1}(x) + \zeta_1(x) - \|\bm{\zeta}\|_{\bm{L}^{\infty}(\Omega)}-\|(u_{h,1}-\chi)^{+}\|_{L^{\infty}(\Gamma_C)},
	\end{align*} which implies
	\begin{align*}
	u_{h,1}(x)-\|(u_{h,1}-\chi)^{+}\|_{L^{\infty}(\Gamma_C)} > u_1(x),
	\end{align*}
	$\implies$
	\begin{align*}
	\chi(x) > u_1(x).
	\end{align*}
	Thus, $supp(z_1) \subset \{u_1< \chi\} $. This concludes the proof.
\end{proof} 
\noindent
As a consequence of Lemma \ref{lem:barrier} and  Lemma \ref{lem:barrier1} we have the following estimate.
\begin{lemma} \label{jhbd}
	Let $\bm{u}$ and $\bm{u}_h$ be the solutions of (\ref{eq:CVI}) and (\ref{eq:DVI}), respectively. There holds,
	\begin{align*}
	\|\bm{u}-\bm{u}_h\|_{\bm{L}^{\infty}(\Omega)} \leq 2\|\bm{\zeta}\|_{\bm{L}^{\infty}(\Omega)} + \|(u_{h,1}-\chi)^{+}\|_{L^{\infty}(\Gamma_C)} + \|(\chi-u_{h,1})^+\|_{L^{\infty}(\Lambda^C_h)} \notag.
	\end{align*}
\end{lemma} 
\noindent
From Lemma \ref{jhbd}, it is evident to bound the term $\|\bm{\zeta}\|_{\bm{L}^{\infty}(\Omega)}$ in order to prove the reliability estimate. Next, we provide the estimate on the maximum norm of $\bm{\zeta}$ in terms of the local error estimators $\eta_i  \quad \forall i=1,2,\cdots,5$. The key ingredient for the supremum norm a posteriori error analysis is the bounds on the Green's matrix for the divergence type operators, for which we refer \cite{dolzmann1995estimates}. Our approach is different than that followed in  the articles \cite{nochetto2003pointwise,nochetto2005fully,nochetto2006pointwise}, as they used the regularized Green's function for the laplacian operator. We follow the technique shown by Demlow \cite{demlow2012pointwise} which varies substantially in several technical details from the previous works \cite{nochetto2003pointwise,nochetto2005fully}. Finally, we state the next lemma which provides the bound on $\|\bm{\zeta}\|_{\bm{L}^{\infty}(\Omega)}$ and discuss the proof in brevity.
\begin{lemma} \label{bound1}
It holds that
	\begin{align} \label{bound2}
	\|\bm{\zeta}\|_{\bm{L}^{\infty}(\Omega)} \lesssim {l_h} \Psi,
	\end{align}
	where $\Psi$ is defined in Theorem \ref{thm:rel}.
\end{lemma}
\begin{proof}
	Let $l \in \{1,2\}$ and $z_0 \in \Omega \setminus \partial \Omega $ {be} such that $	|\zeta^l(z_0)|=\|\bm{\zeta}\|_{\bm{L}^{\infty}(\Omega)} $. Let ${\bm{G}_l^{z_0}}:= l^{th}$ column of the Green's matrix ${\bm{G}^{z_0}}$ defined in \eqref{eq:GR}. In the view of \eqref{eq:Corrector} and \eqref{eq:GR}, the following holds
	\begin{align} \label{eq:LOO}
	\zeta^l(z_0)=	\int_{\Omega} a^{\alpha\beta}_{ij} D_{\beta} G_{jl}(\cdot, z_0)D_{\alpha} \zeta^l~dx &= \tilde{\bm{\mathcal{G}}_h}( \bm{G}_l^{z_0}).
	\end{align}
 To prove the estimate \eqref{bound2}, it is enough to bound the term $\tilde{\bm{\mathcal{G}}_h}( \bm{G}_l^{z_0})$. From \eqref{eq:extGAL}, we {have}
 \begin{align}
 	\tilde{\bm{\mathcal{G}}_h}(\bm{G}_l^{z_0})&= \tilde{a}(\bm{u}-\bm{u}_h,\bm{G}_l^{z_0}) + \langle \bm{\lambda}^*-\bar{\bm{\lambda}}_h, \bm{G}_l^{z_0} \rangle_{-1,1} \hspace{1cm} (\text{using \eqref{4.1})} \nonumber \\ &= L(\bm{G}_l^{z_0}) - \tilde{a}(\bm{u}_h, \bm{G}_l^{z_0}) - \langle \bar{\bm{\lambda}}_h, \bm{G}_l^{z_0} \rangle_{-1,1} \hspace{1cm} (\text{using \eqref{eq:extsigmadef})}  \nonumber  \\  & = \langle \tilde{ \bm{\mathcal{L}}}_{h}, \bm{G}_l^{z_0}\rangle_{-1,1} - \langle \bar{\bm{\lambda}}_h,\bm{G}_l^{z_0} \rangle_{-1,1} \hspace{1.7cm} (\text{using \eqref{extres})} 
 \end{align}
 Exploiting the property of nodal basis function $\underset{p \in \mathcal{N}_h \cup \mathcal{M}_h}{\sum}\phi_p = 1 $,  we find
 \begin{align}
\tilde{\bm{\mathcal{G}}_h}(\bm{G}_l^{z_0}) &= \sum_{k=1}^{2}  \sum_{p \in \mathcal{N}_h \cup \mathcal{M}_h} \langle \tilde{{\mathcal{L}}}_{h,k}, G_{l,k}^{z_0} \phi_p \rangle_{-1,1} -  \sum_{p \in \mathcal{N}_h^C \cup \mathcal{M}_h^C} \langle \tilde{\lambda}_{h,1}, G_{l,1}^{z_0}  \phi_p \rangle_{-1,1}  \hspace{0.7cm} (\text{using Lemma \ref{sign2})}\nonumber  \\ 
 	& = \sum_{k=1}^{2}  \sum_{p \in (\mathcal{N}_h \cup \mathcal{M}_h) \setminus (\mathcal{N}_h^C \cup  \mathcal{M}_h^C)   }  \langle \tilde{{\mathcal{L}}}_{h,k}, G_{l,k}^{z_0} \phi_p \rangle_{-1,1} + \sum_{p \in \mathcal{N}_h^C \cup \mathcal{M}_h^C}  \langle \tilde{{\mathcal{L}}}_{h,2}, G_{l,2}^{z_0} \phi_p \rangle_{-1,1}  \nonumber  \\  & \hspace{0.3cm}+ \sum_{p \in \mathcal{N}_h^C \cup \mathcal{M}_h^C}  \langle \tilde{{\mathcal{L}}}_{h,1}, G_{l,1}^{z_0} \phi_p \rangle_{-1,1}-  \sum_{p \in \mathcal{N}_h^C \cup \mathcal{M}_h^C} \langle \tilde{\lambda}_{h,1}, G_{l,1}^{z_0}  \phi_p \rangle_{-1,1}. \label{4.22}
\end{align}
{
We follow the next step by subtracting \eqref{prop11} and \eqref{prop22} from equation \eqref{4.22} to find
\begin{align*}
\tilde{\bm{\mathcal{G}}_h}(\bm{G}_l^{z_0})& = \sum_{k=1}^{2}  \sum_{p \in (\mathcal{N}_h \cup \mathcal{M}_h) \setminus (\mathcal{N}_h^C \cup  \mathcal{M}_h^C)   }  \langle \tilde{{\mathcal{L}}}_{h,k}, (G_{l,k}^{z_0}-e_p(G_{l,k}^{z_0})) \phi_p \rangle_{-1,1} \nonumber  \\  & \hspace{0.4cm} + \sum_{p \in \mathcal{N}_h^C \cup \mathcal{M}_h^C}  \langle \tilde{{\mathcal{L}}}_{h,2}, (G_{l,2}^{z_0}-e_p(G_{l,2}^{z_0})) \phi_p \rangle_{-1,1}  \nonumber  + \sum_{p \in \mathcal{N}_h^C \cup \mathcal{M}_h^C}  \langle \tilde{{\mathcal{L}}}_{h,1}, (G_{l,1}^{z_0}-e_p(G_{l,1}^{z_0})) \phi_p \rangle_{-1,1}. \nonumber 
\end{align*}
Employing \eqref{rel10} and the representation \eqref{residual}, we have
\begin{align*}
\tilde{\bm{\mathcal{G}}_h}(\bm{G}_l^{z_0}) &= \sum_{k=1}^{2}\sum_{p \in \mathcal{N}_h \cup \mathcal{M}_h}\bigg( \int_{\Omega_p} s_i(\bm{u}_h)((G^{z_0}_{l,k}-e_p(G^{z_0}_{l,k})))\phi_p~dx+ \int_{\Gamma_{p,I}} \cJ^I_{i,e}(\bm{u}_h)(G^{z_0}_{l,k}-e_p(G^{z_0}_{l,k}))\phi_p~ds \bigg) \notag\\
	& \hspace{0.4cm} + \sum_{k=1}^{2} \sum_{p \in \cN_h^{\bar{N}} \cup \cM_h^{\bar{N}}} \int_{\Gamma_{p,N}} \cJ^N_{i,e}(\bm{u}_h) (G^{z_0}_{l,k}-e_p(G^{z_0}_{l,k}))\phi_p~ds \nonumber \\ & \hspace{0.4cm} -\sum_{p \in \mathcal{N}_h^C \cup \mathcal{M}_h^C} \int_{\Gamma_{p,C}} \hat{\sigma_2} (\bm{u}_h) (G^{z_0}_{l,2}-e_p(G^{z_0}_{l,2}))\phi_p~ds
	 \notag \\& \hspace{0.4cm}
 - \sum_{p \in \mathcal{N}_h^C \cup \mathcal{M}_h^C} \int_{\Gamma_{p,C}} \hat{\sigma_1} (\bm{u}_h) (G^{z_0}_{l,1}-e_p(G^{z_0}_{l,1})) \phi_p~ds. \notag 
 \end{align*}
 The H\"older's inequality then yields the following bound
 \begin{align}
\tilde{\bm{\mathcal{G}}_h}(\bm{G}_l^{z_0})&\lesssim \sum_{k=1}^{2} \sum_{p \in \mathcal{N}_h \cup \mathcal{M}_h} \Big(h^2_p\|s_i(\bm{u}_h)\|_{L^\infty(\Omega_p)} h^{-2}_p\|G^{z_0}_{l,k}-e_p(G^{z_0}_{l,k})\|_{L^1(\Omega_p)} \Big) \notag\\
	& \hspace{0.4cm}+ \sum_{k=1}^{2} \sum_{p \in \mathcal{N}_h \cup \mathcal{M}_h} \Big( h_p\|\cJ^I_{i,e}(\bm{u}_h)\|_{L^\infty(\Gamma_{p,I})}h^{-1}_p\|G^{z_0}_{l,k}-e_p(G^{z_0}_{l,k})\|_{L^1(\Gamma_{p,I})}\Big) \notag \\
	& \hspace{0.4cm}+ \sum_{k=1}^{2} \sum_{p \in \cN_h^{\bar{N}} \cup \cM_h^{\bar{N}}} \Big(h_p\|\cJ^N_{i,e}(\bm{u}_h)\|_{L^\infty(\Gamma_{p,N})}h^{-1}_p\|G^{z_0}_{l,k}-e_p(G^{z_0}_{l,k})\|_{L^1(\Gamma_{p,N})} \Big) \notag \\
	& \hspace{0.4cm} + \sum_{p \in \cN^C_h \cup  \cM^C_h} \Big(h_p\|\hat{\sigma_2} (\bm{u}_h)\|_{L^\infty(\Gamma_{p,C})} h^{-1}_p\|G^{z_0}_{l,2}-e_p(G^{z_0}_{l,2})\|_{L^1(\Gamma_{p,C})}  \Big) \notag \\
	& \hspace{0.4cm}+ \sum_{p\in \cN^C_h \cup \cM^C_h} \Big(h_p\|\hat{\sigma_1} (\bm{u}_h)\|_{L^\infty(\Gamma_{p,C})} h^{-1}_p\|G^{z_0}_{l,1}-e_p(G^{z_0}_{l,1})\|_{L^1(\Gamma_{p,C})}  \Big) . 
\end{align}
}
Using Lemma \ref{Lemmmma} and Lemma \ref{approxprop}, we observe
\begin{align} \label{DFA}
\big|\tilde{\bm{\mathcal{G}}_h}( \bm{G}_l^{z_0})\big|  & \lesssim \Psi \sum_{k=1}^{2} \Big( \sum_{p \in \cN_h \cup \cM_h} \Big( h^{-2}_p \|(G^{z_0}_{l,k}-e_p(G^{z_0}_{l,k}))\|_{L^1(\Omega_p)} \notag\\  & \hspace{0.7cm}+ h^{-1}_p \|\nabla((G^{z_0}_{l,k}-e_p(G^{x_0}_{l,k})))\|_{L^1(\Omega_p)} \Big) \Big) \notag \\ & \lesssim \Psi \bigg(\sum_{k=1}^{2} \sum_{T \in \cT_h} h^{-1}_T |G^{z_0}_{l,k} |_{1,1,T}\bigg).
\end{align} 
It suffices to bound the term on the right hand side in \eqref{DFA} to prove the estimate \eqref{bound2}. Let $T^* \in \cT_h$ be such that $z_0 \in T^*$ and let $\Theta_0$ be a patch around the element $T^*$, i.e., set of all element touching $T^*$. In our estimates, we introduce the dyadic decomposition of the finite element triangulation $\cT_h = (\cT_h \setminus \Theta_0) \cup (\cT_h \cap \Theta_0)$. This decomposition will help us to use the regularity estimates mentioned in Lemma \ref{regularity}. Due to the assumption of shape regularity, there exist constants $a$ and $b$ with $a > b\geq 0$ such that
	\begin{align*}
\Theta_0 &\subset B_0:=B_{a{h}_{z_0}}(z_0) \\ (\cT_h \setminus \Theta_0) &\subset (\cT_h \setminus B_1):= {(\cT_h \setminus B_{bh_{{z}_0}}(z_0))},
\end{align*} 
where $B_x(y)$ is the ball with radius $y$ and center $x$ and $h_x: \Omega \rightarrow \mathbb{R}$ is defined by $h_x:=\text{diam}~T~ \text{if}~ x \in T.$ Next, we use the notations defined in the last paragraph to bound the right hand side term in \eqref{DFA}. Let us fix some $j \in \{1,2\}$, then,
\begin{align} \label{EE}
\sum_{T \in \cT_h} h^{-1}_T |G^{z_0}_{l,k}|_{1,1,T} & = \sum_{T \in \cT_h \cap \Theta_0}  h^{-1}_T|G^{z_0}_{l,k}|_{W^{1,1}(T)} + \sum_{T \in \cT_h \setminus \Theta_0}  h^{-1}_T|G^{z_0}_{l,k}|_{W^{1,1}(T)} \notag \\ &  \lesssim h^{-1}_{z_0}|G^{z_0}_{l,k}|_{W^{1,1}(B_0)} + {\sum_{T \in \cT_h \setminus \Theta_0}  h^{-1}_T|G^{z_0}_{l,k}|_{W^{1,1}(T)}}.
\end{align}
\par
\noindent
We try to bound the term $h^{-1}_{z_0}|G^{z_0}_{l,k}|_{W^{1,1}(B_0)} $ of \eqref{EE} and skip the proof for the second term. Using the H\"older's inequality for $q = \frac{4}{3} < 2$ and equation \eqref{bound4}, we have 
\begin{align}
h^{-1}_{z_0}|G^{z_0}_{l,k}|_{W^{1,1}(B_0)} & \lesssim h^{-1}_{z_0} \Big( |B_{h^2_{z_0}}(z_0)|^{1-\frac{1}{q}}\|\nabla G^{z_0}_{l,k}\|_{L^q(B_{h^2_{z_0}}(z_0) \cap B_0)} + |G^{z_0}_{l,k}|_{W^{1,1}( B_0 \setminus B_{h^2_{z_0}}(z_0)) } \Big) \notag \\ & \lesssim 1 + h^{-1}_{z_0} |G^{z_0}_{l,k}|_{W^{1,1}( B_0 \setminus B_{h^2_{z_0}}(z_0)) }. \label{5.20}
\end{align}
Next, we define $ \Omega_i:= \{ z \in \Omega : p_i < |z-z_0|\leq p_{i+1}\}$ where  $p_i = 2^{i-1}h^2_{z_0}$, $i=0,1,....~$. We then notice that $B_0 \setminus B_{h^2_{z_0}}(z_0)= \bigcup\limits_{i=1}^{\mathcal{J}} \Omega_i$ for some $\cJ \leq C \text{ln}(\frac{1}{h_{z_0}})$ by an annular dyadic decomposition of $B_0 \setminus B_{h^2_{z_0}}(z_0)$. Also assume $ \Omega_{i'}= \Omega_{i-1} \cup \Omega_i \cup \Omega_{i+1} $. Using equation \eqref{eq:BGRR}, Cacciopoli-Leray inequality \cite{ambrosio2015lecture} and H\"older's inequality, we have
{
\begin{align*}
|G^{z_0}_{l,k}|_{W^{1,1}(B_0 \setminus B_{h^2_{z_0}}(z_0))} & \lesssim \sum_{i=1}^{\cJ} p_i \|\nabla G^{z_0}_{l,k}\|_{L^2(\Omega_i)} \lesssim \sum_{i=1}^{\cJ} \| G^{z_0}_{l,k}\|_{L^2(\Omega_{i'})} \\ & \lesssim \sum_{i=0}^{\cJ} {p_i}\| G^{z_0}_{l,k}\|_{L^{\infty}(\Omega_{i})} \lesssim \sum_{i=0}^{\cJ} {p_i} |\text{log}(p_i)| \\ 
%%& \lesssim \sum_{i=0}^{\cJ} 2^{i-1}h^2_{z_0} |\text{log} (2^{i-1}h^2_{z_0})| \lesssim \sum_{i=0}^{\cJ} 2^{i-1}h^2_{z_0} %\text{log} (2^{i-1}) + 2\text{log}(h_{z_0})| \\ 
%& =   \sum_{i=0}^{\cJ} 2^{i-1} h^2_{z_0} \bigg[\big|(i-1)\text{log} (2) - 2\text{log}(\frac{1}{h_{z_0}}) \big| \bigg]   \\ 
%& \lesssim |\text{log} \frac{1}{h_{z_0}} |\sum_{i=0}^{\cJ} 2^{i-1} h^2_{z_0}\\ 
& \lesssim h_{z_0} |\text{log} \frac{1}{h_{z_0}}| 
%\\ & \lesssim h_{z_0} |\text{log}(h_{min})|^2
\end{align*}}
\par
\noindent
Therefore, using equation \eqref{5.20}, we have
\begin{align} \label{pp}
h^{-1}_{z_0}|G^{z_0}_{l,k}|_{W^{1,1}(B_0)} \lesssim {1+|\text{log}(h_{min})|}.
\end{align}
{
 	Next, we deal with the second term on the right hand side of \eqref{EE}. Let us introduce the annular decomposition of $\cT_h \setminus B_1$. We define $ \Omega_i:= \{ z \in \Omega : p_i < |z-z_0|< p_{i+1}\}$ where  $p_i = 2^{i-1}bh_{z_0}$, $i=0,1,....~$. Let $\mathcal{J}$ be such that $\cT_h \setminus \Theta_0 \subseteq \bigcup\limits_{i=1}^{\mathcal{J}} \Omega_{i}$ which yields $J \lesssim \text{log}(\frac{1}{h_{z_0}})$ (see Lemma 2.1, \cite{kashiwabara2020pointwise}). Also assume $ \Omega_{i'}= \Omega_{i-1} \cup \Omega_i \cup \Omega_{i+1} $. Using Cauchy Schwartz inequality, \eqref{eq:BGRR}, Cacciopoli-Leray inequality \cite{ambrosio2015lecture} and H\"older's inequality, we have
 	\begin{align} \label{ppp}
 	\sum_{T \in \cT_h \setminus \Theta_0}  h^{-1}_T|G^{z_0}_{l,k}|_{W^{1,1}(T)} 
% 	& = \sum_{T \in \cT_h \setminus \Theta_0}  h^{-1}_T \int_T \nabla G^{z_0}_{l,k} dx \notag \\ 
 	& \lesssim \sum_{T \in \cT_h \setminus \Theta_0}   |G^{z_0}_{l,k}|_{W^{1,2}(T)} 
% 	\notag \\ &
 	\lesssim \sum_{i=1}^{\mathcal{J}}   |G^{z_0}_{l,k}|_{W^{1,2}(\Omega_{i})} \notag  \\ 
 	&\lesssim \sum_{i=1}^{\cJ}   p_i^{-1} |G^{z_0}_{l,k}|_{L^{2}(\Omega_{i'})} 
% 	\notag \\ &
 	 \lesssim \sum_{i=0}^{\cJ}  |G^{z_0}_{l,k}|_{L^{\infty}(\Omega_{i})}  \notag \\ 
% 	& \lesssim \sum_{i=0}^{\cJ} |\text{log}(p_i)| \notag =\sum_{i=0}^{\cJ} |\text{log}(2^{i-1}bh_{z_0})| \notag \\ 
% 	& =\sum_{i=0}^{\cJ} |(i-1)\text{log}(2) + \text{log}(bh_{z_0})| \notag \\ 
 	& \lesssim \cJ | \text{log}(\frac{1}{h_{z_0}})| 
% 	\notag \\ & 
 	\lesssim |\text{log}(h_{min})|^2.
 	\end{align}  
 	Using \eqref{EE}, \eqref{pp} and \eqref{ppp}, we have 
 	\begin{align}\label{eq:Ghest}
\sum_{T \in \cT_h} h^{-1}_T |G^{z_0}_{l,k}|_{1,1,T} \lesssim 1+|\text{log}(h_{min})|^2.
 	\end{align}}
 \par
\noindent
Hence, we obtain the desired result of this lemma using \eqref{DFA}, \eqref{eq:LOO} and \eqref{eq:Ghest}.
\end{proof}
\noindent
Next, we derive the following bound on the Galerkin functional $\bm{\mathcal{G}}_h$.
\begin{lemma} \label{bound}
	It holds that
	\begin{align} \label{bound3}
	\|\bm{\mathcal{G}}_h\|_{-2,\infty,\Omega} \lesssim \Psi.
	\end{align}
	
\end{lemma}
\begin{proof}
	Let $\bm{v}\in \bm{V} \cap \bm{U}$. We follow the same steps as in the Lemma \ref{bound1} and using Cauchy-Schwarz inequality and {the Poincar$\acute{e}$-type inequality $|\phi|_{1,2,\Omega} \lesssim |\phi|_{2,1,\Omega}$} \cite[p 522]{nochetto2006pointwise} to derive the following
	\begin{align}
	\bm{\mathcal{G}}_h( \bm{v})  & \lesssim  \Psi \Big( \sum_{i=1}^{2} \Big(  \sum_{p \in \cN_h \cup \cM_h} h^{-2}_p \|(v_i-e_p(v_i))\|_{L^1(\Omega_p)} + h^{-1}_p \|\nabla(v_i-e_p(v_i))\|_{L^1(\Omega_p)} \Big)   \Big) \notag  \\ & \lesssim \Psi \Big(\sum_{i=1}^{2} \sum_{T \in \cT_h} h^{-1}_T |v_i |_{W^{1,1}(T)}\Big) \notag \\&
	\lesssim \Psi \Big(\sum_{i=1}^{2} \sum_{T \in \cT_h}  |v_i|_{W^{1,2}(T)}\Big) \notag \\&
	\lesssim \Psi \Big(\sum_{i=1}^{2} \sum_{T \in \cT_h} |v_i |_{W^{2,1}(T)}\Big) \notag.
	\end{align}	
	Finally, in view of the definition \eqref{norm}, we obtain  the estimate \eqref{bound3}.
\end{proof}
\noindent
Now we proceed to provide a proof of  Theorem \ref{thm:rel}.
\break
\textbf{Proof of Theorem \ref{thm:rel}}
{
 Using the Lemma \ref{jhbd} and the estimate for $\|\bm{\zeta}\|_{\bm{L}^{\infty}(\Omega)}$ from Lemma \ref{bound1}, we get the following reliability  estimate for $\|\bm{u}-\bm{u}_h\|_{\bm{L}^{\infty}(\Omega)}$. Next, in order to estimate $\|\bm{\lambda}-\bar{\bm{\lambda}}_h\|_{-2,\infty,\Omega}$, we let $\bm{v}\in \bm{V} \cap \bm{U}$. Using \eqref{eq:extGAL}, integration by parts, definition of $a(\cdot,\cdot)$ and H\"older's inequality, we conclude}
	\begin{align} \label{BOUND}
	\langle \bm{\mathcal{G}}_h, \bm{v} \rangle_{-1,1} &= \int_{\Omega} \bm{\epsilon}(\bm{u}-\bm{u}_h) : \bm{\sigma}(\bm{v}) dx +   \langle \bm{\lambda}-\bar{\bm{\lambda}}_h , \bm{v} \rangle_{-1,1} \notag \\&= - \int_{\Omega} (\bm{u}-\bm{u}_h) \cdot \text{div} \bm{\sigma}(\bm{v})  dx +   \langle \bm{\lambda}-\bar{\bm{\lambda}}_h , \bm{v} \rangle_{-1,1} \hspace{2cm} \notag \\& \lesssim \|\bm{u}-\bm{u}_h\|_{\bm{L}^{\infty}(\Omega)} |\bm{v}|_{\bm{W}^{2,1}(\Omega)} + \|\bm{\lambda}-\bar{\bm{\lambda}}_h\|_{-2,\infty,\Omega} |\bm{v}|_{\bm{W}^{2,1}(\Omega)}.
	\end{align}
	\noindent
	Thus, using the definition \eqref{norm}, we find
	\begin{align} \label{oo}
	\|\bm{\lambda}-\bar{\bm{\lambda}}_h\|_{-2,\infty,\Omega} \lesssim \|\bm{u}-\bm{u}_h\|_{\bm{L}^{\infty}(\Omega)} + \|\bm{\mathcal{G}}_h\|_{-2,\infty,\Omega}.
	\end{align}
%	 A use of bound on the term $\|\bm{u}-\bm{u}_h\|_{\bm{L}^{\infty}(\Omega)}$ (equation \eqref{rel1}) yields the bound of the second term of the right hand side of \eqref{oo}. 
	 Hence, we have the desired reliability estimate for the error in contact force density by using the Lemma \ref{bound} and the realibity estimate for $\|\bm{u}-\bm{u}_h\|_{\bm{L}^{\infty}(\Omega)}$.
\subsection{Efficiency of the error estimator} In this section, we show that the residual contributions of the error estimator $\eta_h$, defined in \eqref{est}, is bounded above by the error plus data oscillations. Standard bubble function techniques \cite{verfurth1996review} are used to prove the efficiency estimates of the error estimator in this section. In the analysis below, we denote for any $\bm{v} \in \bm{V}_h$ and $T \in \cT_h$, $\bm{\epsilon}_h(\bm{v})$ as $\bm{\epsilon}_h(\bm{v})|_T= \bm{\epsilon(\bm{v})}$ on $T$ and $\bm{\sigma}_h(\bm{v})= 2 \mu \bm{\epsilon}_h(\bm{v}) + \zeta~ tr (\bm{\epsilon}_h(\bm{v}))  I$. We denote the term $Osc(\bm{f},p):=h^2_p \|\bm{f}-\bar{\bm{f}}\|_{\bm{L}^{\infty}(\Omega_p)}$ to be the data oscillation of the load vector $\bm{f}$, where  $\bar{\bm{f}}$ is the piecewise constant approximation of $\bm{f}$.  The oscillation term related to the Neumann data $\bm{g}$ is defined by, $ Osc(\bm{g},p):= h_p\|\bm{g}-\bar{\bm{g}}\|_{\bm{L}^{\infty}(\Gamma_{p,N})}$,  where $\bar{\bm{g}}$ is piecewise constant approximation of $\bm{g}$.
\begin{remark}
	Let $p \in \cN^C_h \cup \cM^C_h$ and $e \in \Gamma_{p,C}$ be such that $ \langle \bm{\lambda}_h, \psi_p \bm{e}_1 \rangle_h >0 $, then
	\begin{align} \label{eff6}
	\|(\chi-u_{h,1})^+\|_{L^{\infty}(e)} \leq \|(\chi-\chi_h)^+\|_{L^{\infty}(e)} + \|(\chi_h-u_{h,1})^+\|_{L^{\infty}(e)}.
	\end{align}
	For the smooth obstacle function, the first term of the right hand side of the last estimate will be of higher order. The efficiency of the second term of the right hand side of \eqref{eff6} is still less clear due to the quadratic nature of the discrete solution $\bm{u}_h$. This subject will be pursued in the future.
\end{remark}
\noindent
We collect the main result of this section in the next theorem.
\begin{theorem}
 It holds that 
		\begin{align} 
		&\eta_{1,p} + \|(u_{h,1}-\chi)^{+}\|_{L^{\infty}(\Gamma_C)} \lesssim \|\bm{u}-\bm{u}_h\|_{\bm{L}^{\infty}(\Omega_p)} + 	\|\bm{\lambda}-\bar{\bm{\lambda}}_h\|_{-2,\infty,\Omega_p} \nonumber \\ &\hspace{5.2cm} +Osc(\bm{f},p) \quad \forall~ p \in \mathcal{N}_h \cup \mathcal{M}_h, \label{eff1}\\
			&\eta_{2,p}  \lesssim  \|\bm{u}-\bm{u}_h\|_{\bm{L}^{\infty}(\Omega_p)} + 	\|\bm{\lambda}-\bar{\bm{\lambda}}_h\|_{-2,\infty,\Omega_p} + Osc(\bm{f},p)  \quad \forall~ p \in \mathcal{N}_h \cup \mathcal{M}_h, \label{eff2}\\
				&\eta_{3,p}  \lesssim  \|\bm{u}-\bm{u}_h\|_{\bm{L}^{\infty}(\Omega_p)} + 	\|\bm{\lambda}-\bar{\bm{\lambda}}_h\|_{-2,\infty,\Omega_p} + Osc(\bm{f},p) + Osc(\bm{g},p) \quad \forall~ p \in \mathcal{N}_h^N \cup \mathcal{M}_h^N, \label{eff3}\\
			&	\eta_{4,p}  \lesssim \|\bm{u}-\bm{u}_h\|_{\bm{L}^{\infty}(\Omega_p)} + 	\|\bm{\lambda}-\bar{\bm{\lambda}}_h\|_{-2,\infty,\Omega_p} + Osc(\bm{f},p) \quad \forall~ p \in \mathcal{N}_h^C \cup \mathcal{M}_h^C, \label{eff4} \\
			&	\eta_{5,p}  \lesssim  \|\bm{u}-\bm{u}_h\|_{\bm{L}^{\infty}(\Omega_p)} + 	\|\bm{\lambda}-\bar{\bm{\lambda}}_h\|_{-2,\infty,\Omega_p} + Osc(\bm{f},p) \quad \forall~ p \in \mathcal{N}_h^C \cup \mathcal{M}_h^C. \label{eff5}
		\end{align}
\end{theorem}
\begin{proof}
		{$\bullet$} The lower bound on the second term $ \|(u_{h,1}-\chi)^{+}\|_{L^{\infty}(\Gamma_C)}$ follows immediately as $u_1 \leq \chi$ on $\Gamma_C$. Next, we bound the term $\eta_{1,p}=h_p^2 \|\bm{s}(\bm{u}_h)\|_{\bm{L}^{\infty}(\Omega_p)}$ where $p \in \mathcal{N}_h \cup \mathcal{M}_h.$ Let $T \in \Omega_p$ and assume $\phi_T \in \mathbb{P}_3(T) \cap W^{2,1}_0(T)$ be the bubble function \cite{verfurth1996review} which takes unit value at the barycenter of $T$ and zero value on $\partial T$. Let $\bm{\beta}_T=\phi_T(\bar{\bm{f}}+\text{div} \bm{\sigma}_h(\bm{u}_h))$ on $T$ and set $\bm{\beta} \in \bm{H}^1(\Omega) \cap \bm{W}^{2,1}_0(\Omega)$ to be an extension of $\bm{\beta}_T$ to whole $\Omega$ by zero. Using the equivalence of norms in finite dimensional spaces on a reference element followed by scaling, we have an existence of positive constants $C_1$ and $C_2$ such that
		\begin{align} \label{equi}
 C_1 \int_T (\text{div} \bm{\sigma}_h(\bm{u}_h) + \bar{\bm{f}}) \cdot \bm{\beta}~dx \leq \|\text{div} \bm{\sigma}_h(\bm{u}_h) + \bar{\bm{f}}\|^2_{\bm{L}^2(T)} \leq C_2 \int_T (\text{div} \bm{\sigma}_h(\bm{u}_h) + \bar{\bm{f}}) \cdot \bm{\beta}~dx.
\end{align}
Using H\"older's inequality, Lemma \ref{Lemama} and together with the structure of $\bm{\beta}_T$, we get
\begin{align} \label{estt}
|\bm{\beta}_T|_{\bm{W^{2,1}}(T)} & \lesssim h^{-2}_T \|\bm{\beta}_T\|_{\bm{L}^1(T)} \lesssim  \|\bm{\beta}_T\|_{\bm{L}^{\infty}(T)} \lesssim \|(\text{div} \bm{\sigma}_h(\bm{u}_h) + \bar{\bm{f}})\|_{\bm{L}^{\infty}(T)}.
\end{align}
and 
\begin{align} \label{estt1}
\|\bm{\beta}_T\|_{\bm{L}^1(T)} &\lesssim h^2_T \|\bm{\beta}_T\|_{\bm{L}^{\infty}(T)}  \lesssim h^2_T \|(\text{div} \bm{\sigma}_h(\bm{u}_h) + \bar{\bm{f}})\|_{\bm{L}^{\infty}(T)}.
\end{align}
A use of estimates \eqref{estt}, \eqref{estt1}, H\"older's inequality, integration by parts, Lemma \ref{Lemama}, equations \eqref{equi} and $\langle \bar{\bm{\lambda}}_h, \bm{\beta} \rangle_{-1,1}=0$ yields
	\begin{align} \label{o}
h^4_T\|\text{div} \bm{\sigma}_h(\bm{u}_h) + \bar{\bm{f}}\|^2_{\bm{L}^{\infty}(T)} & \lesssim h^{2}_T\|\text{div} \bm{\sigma}_h(\bm{u}_h) + \bar{\bm{f}}\|^2_{\bm{L}^2(T)} \notag \\ &\lesssim h^{2}_T \int_T (\text{div} \bm{\sigma}_h(\bm{u}_h) + \bar{\bm{f}}) \bm{\beta} ~dx \notag \\ &= h^{2}_T \Big(\int_T (\bm{\bar{f}}- \bm{f})\bm{\beta} ~dx + \int_T (\bm{f}+ \text{div} \bm{\sigma}_h(\bm{u}_h))\bm{\beta} ~dx\Big) \notag \\ & =h^{2}_T \Big( \int_T (\bm{\bar{f}}- \bm{f})\bm{\beta} ~dx + \langle \bm{\mathcal{G}}_h ,\bm{\beta} \rangle_{-1,1}  \Big)\hspace{1cm} \notag\\ & \lesssim h^{2}_T \Big(\|\bm{\bar{f}}- \bm{f}\|_{\bm{L}^{\infty}(T)} \|\bm{\beta}_T\|_{\bm{L}^1(T)} + \|\bm{\mathcal{G}}_h\|_{-2,\infty,T} |\bm{\beta}_T|_{\bm{W^{2,1}}(T)} \Big) \notag  \\ & \lesssim h^2_T \Big( \|\bm{\mathcal{G}}_h\|_{-2,\infty,T} + h^2_T\|\bm{\bar{f}}- \bm{f}\|_{\bm{L}^{\infty}(T)} \Big)\|\text{div} \bm{\sigma}_h(\bm{u}_h) + \bar{\bm{f}}\|_{\bm{L}^{\infty}(T)} \notag
\end{align} 
and finally, we have the desired estimate \eqref{eff1} using equations \eqref{norm} and \eqref{BOUND}.\\
{$\bullet$} Let $p \in \mathcal{N}_h \cup \mathcal{M}_h$ and $e \in \Gamma_{p,I}$ be an interior edge sharing the elements $T_1$ and $T_2$. We denote $\tau_e= T_1 \cup T_2$ and $\phi_e \in W^{2,1}_0(\tau_e) \cap \mathbb{P}_4(\tau_e)$ be a bubble function which assumes unit value at the center of $e$. Define $\bm{q} \in \bm{H}^2_0(\Omega)$ to be an extension of $\bm{q}_e=\phi_e \bm{\cJ}^I(\bm{u}_h)$ to $\bar{\Omega}$ by zero. From the equivalence of norms and scaling on the reference element, we have the existence of two positive constants $C_1$ and $C_2$ such that
	\begin{align}
 C_2 \int_e \bm{\cJ}_e^I(\bm{u}_h) \cdot \bm{q}_e~ds \leq \|\bm{\cJ}_e^I(\bm{u}_h)\|^2_{\bm{L}^2(e)} \leq C_1 \int_e \bm{\cJ}_e^I(\bm{u}_h) \cdot \bm{q}_e~ds.
\end{align}
With the help of Lemma \ref{Lemama}, integration by parts, we have
\begin{align} \label{jhc}
h^2_e \|\bm{\cJ}_e^I(\bm{u}_h)\|^2_{\bm{L}^{\infty}(e)} & \lesssim h_e \|\bm{\cJ}_e^I(\bm{u}_h)\|^2_{\bm{L}^2(e)} \notag \\ & \lesssim h_e \Big(\int_e \bm{\cJ}_e^I(\bm{u}_h) \bm{q}_e~ds \Big)  \notag \\ & = h_e  \Big(\int_{\tau_e} \text{div} \bm{\sigma}_h(\bm{u}_h) \bm{q}_e~dx+ \int_{\tau_e} \bm{\sigma}_h(\bm{u}_h) : \bm{\epsilon}(\bm{q}_e)~dx\Big) \notag \\ & = h_e  \Big(\int_{\tau_e} \bm{s}(\bm{u}_h) \cdot \bm{q}_e~dx - \int_{\tau_e} \bm{f} \cdot \bm{q}_e~dx+ \int_{\tau_e} \bm{\sigma}_h(\bm{u}_h) : \bm{\epsilon}(\bm{q}_e)~dx \Big). \notag 
\end{align}
Next, using  $\langle \bar{\bm{\lambda}}_h, \bm{q} \rangle_{-1,1}=0$ and equations \eqref{eq:sigmadef},  \eqref{eq:extGAL}, we deduce
\begin{align}
h^2_e \|\bm{\cJ}_e^I(\bm{u}_h)\|^2_{\bm{L}^{\infty}(e)} & \lesssim h_e  \Big( \int_{\Omega} \bm{s}(\bm{u}_h) \bm{q}~dx -\langle \bm{\mathcal{G}}_h , \bm{q} \rangle_{-1,1}  \Big).
\end{align}
A use of H\"older's inequality and structure of $\bm{q}_e$ gives
\begin{align}
h^2_e \|\bm{\cJ}_e^I(\bm{u}_h)\|^2_{\bm{L}^{\infty}(e)} & \lesssim h_e \Big( \|\bm{s}(\bm{u}_h)\|_{\bm{L}^{\infty}(\tau_e)} \|\bm{q}_e\|_{\bm{L}^1(\tau_e)} + \|\bm{\mathcal{G}}_h\|_{-2,\infty,\tau_e} |\bm{q}_e|_{\bm{W}^{2,1}(\tau_e)} \Big) \notag \\ & \lesssim h_e \Big(h^2_e \|\bm{s}(\bm{u}_h)\|_{\bm{L}^{\infty}(\tau_e)}  \|\bm{q}_e\|_{\bm{L}^{\infty}(\tau_e)} + \|\bm{\mathcal{G}}_h\|_{-2,\infty,\tau_e} \|\bm{q}_e\|_{\bm{L}^{\infty}(\tau_e)}  \Big) \notag  \\& \lesssim h_e \Big( \|\bm{\mathcal{G}}_h\|_{-2,\infty,\tau_e} + h_e^2\|\bm{s}(\bm{u}_h)\|_{\bm{L}^{\infty}(\tau_e)}   \Big)\|\bm{\cJ}_e^I(\bm{u}_h)\|_{\bm{L}^{\infty}(e)}. \notag
\end{align}
Finally, we have the desired estimate using bounds from \eqref{eff1} and $\|\bm{G}_h\|_{-2,\infty,\tau_e}$ from equation \eqref{BOUND}.\\
	{$\bullet$} For $p \in \mathcal{N}_h^N \cup \mathcal{M}_h^N$, let $e \in \Gamma_{p,N}$ be a Neumann edge sharing the element $T$. We define the bubble function $\psi_e \in \mathbb{P}_2(T) \in W^{2,1}_0(T)$ corresponding to edge which is zero on $\partial T \setminus e$ and assumes unit value at the midpoint of $e$. Let $\bm{\phi}_e = \psi_e (\bm{\sigma}(\bm{u}_h)\bm{n}- \bar{\bm{g}})$ and $\bm{\phi} \in \bm{H}^1_0(\Omega)$ to be an extension of $\bm{\phi}_e$ by zero outside $T$. By the equivalence of norms in finite dimensional normed spaces and scaling arguments, there exists a positive constant $C_1$ such that
	\begin{align} \label{esti}
	 \|\bm{\sigma}(\bm{u}_h)\bm{n}- \bar{\bm{g}}\|^2_{\bm{L}^2(e)} & \leq C_1 \int_e (\bm{\sigma}(\bm{u}_h)\bm{n}- \bar{\bm{g}}) \bm{\phi}_e~ds \notag \\ & \lesssim \int_e (\bm{\sigma}(\bm{u}_h)\bm{n}- \bm{g}) \bm{\phi}_e~ds + \int_e (\bm{g}-\bar{\bm{g}}) \bm{\phi}_e~ds.
	 \end{align}
A use of inverse inequality (Lemma \ref{Lemama}), estimate \eqref{esti}, $\langle \bar{\bm{\lambda}}_h, \bm{\phi} \rangle_{-1,1}=0$ and integration by parts yields
\begin{align}
h^2_e \|\bm{\sigma}(\bm{u}_h)\bm{n}- \bar{\bm{g}}\|^2_{\bm{L}^{\infty}(e)} & \lesssim h_e \|\bm{\sigma}(\bm{u}_h)\bm{n}- \bar{\bm{g}}\|^2_{\bm{L}^2(e)} \notag \\  & = h_e  \Bigg(\int_{T} \text{div} \bm{\sigma}_h(\bm{u}_h) \bm{\phi}_e~dx+ \int_{T} \bm{\sigma}_h(\bm{u}_h) : \bm{\epsilon}(\bm{\phi}_e)~dx - \int_e \bm{g} \bm{\phi}_e~ds \notag \\ & \hspace{1cm}+ \int_e (\bm{g}-\bar{\bm{g}}) \bm{\phi}_e~ds\Bigg) \notag \\ & = h_e \Bigg( \int_{\Omega} \bm{s}(\bm{u}_h) \bm{\phi}~dx -\langle \bm{\mathcal{G}}_h , \bm{\phi} \rangle_{-1,1}  + \int_e (\bm{g}-\bar{\bm{g}}) \bm{\phi}_e~ds \Bigg).
\end{align}
Inserting the structure of $\bm{\phi}_e$ and using H\"older's inequality, we deduce
\begin{align}
h^2_e \|\bm{\sigma}(\bm{u}_h)\bm{n}- \bar{\bm{g}}\|^2_{\bm{L}^{\infty}(e)}  & \lesssim h_e \Big( \|\bm{s}(\bm{u}_h)\|_{\bm{L}^{\infty}(T)} \|\bm{\phi}_e\|_{\bm{L}^1(T)} + \|\bm{\mathcal{G}}_h\|_{-2,\infty,T} |\bm{\phi}_e|_{\bm{W}^{2,1}(T)} \notag \\ & \hspace{2cm}+  \|\bm{g}-\bar{\bm{g}}\|_{\bm{L}^{\infty}(e)}  \|\bm{\phi}_e\|_{\bm{L}^1(e)} \Big) \notag \\ & \lesssim h_e \Big(h^2_e \|\bm{s}(\bm{u}_h)\|_{\bm{L}^{\infty}(T)}  \|\bm{\phi}_e\|_{\bm{L}^{\infty}(T)} + \|\bm{\mathcal{G}}_h\|_{-2,\infty,T} \|\bm{\phi}_e\|_{\bm{L}^{\infty}(T)} \notag \\ & \hspace{1cm }+ h_e\|\bm{g}-\bar{\bm{g}}\|_{\bm{L}^{\infty}(e)}  \|\bm{\phi}_e\|_{\bm{L}^{\infty}(e)}   \Big) \notag  \\& \lesssim h_e \Big( \|\bm{\mathcal{G}}_h\|_{-2,\infty,T} + h_T^2\|\bm{s}(\bm{u}_h)\|_{\bm{L}^{\infty}(T)}  \notag \\ & \hspace{2cm}+ h_e \|\bm{g}-\bar{\bm{g}}\|_{\bm{L}^{\infty}(e)}  \Big)\|\bm{\sigma}(\bm{u}_h)\bm{n}- \bar{\bm{g}}\|_{\bm{L}^{\infty}(e)}. \notag
\end{align}
Hence, we have the desired estimate \eqref{eff3} using \eqref{BOUND} and \eqref{eff1}.\\
{$\bullet$} We omit the proof for \eqref{eff4} as it follows similarly using the bubble functions technique. Next, we prove \eqref{eff5}. Assume $p \in \mathcal{N}_h^C \cup \mathcal{M}_h^C$ and $e \in \Gamma_{p,C}$ be the corresponding edge sharing the triangle $T$. We follow the same path as in the article \cite{krause2015efficient} to have the suitable bubble function ${b}_e$ such that $e_p(b_e)=0,$ where $e_p(\cdot)$ is the scalar defined in equation (\ref{eq:DPD}).
	\begin{align*} 
\langle \bm{G}_h, b_e\bm{e}_1 \rangle_{-1,1} &=a(\bm{u}-\bm{u}_h,b_e\bm{e}_1)+\langle \bm{\lambda}-\bar{\bm{\lambda}}_h, b_e\bm{e}_1 \rangle_{-1,1}  \notag\\
&= (\bm{f},b_e\bm{e}_1)-a(\bm{u}_h,b_e\bm{e}_1)+\langle \bm{g}, b_e\bm{e}_1 \rangle_{\Gamma_N}- \langle \bar{\bm{\lambda}}_h, b_e\bm{e}_1 \rangle_{-1,1}  \notag\\ & = \Bigg(\int_{\Omega} \bm{s}(\bm{u}_h) \cdot b_e\bm{e}_1~dx - \int_e ( \hat{\bm{\sigma}}(\bm{u}_h)- \bm{g}) \cdot b_e\bm{e}_1 ~ds   -  \int_e \hat{\bm{\sigma}}(\bm{u}_h)\cdot b_e\bm{e}_1~ds \Bigg).
\end{align*}
Therefore, we have
\begin{align} \label{eqeq}
 \int_e \hat{\bm{\sigma}}(\bm{u}_h)\cdot b_e\bm{e}_1~ds =  \int_{\Omega} \bm{s}(\bm{u}_h) \cdot b_e\bm{e}_1~dx - \int_e ( \hat{\bm{\sigma}}(\bm{u}_h)- \bm{g}) \cdot b_e\bm{e}_1 ~ds - \langle \bm{\mathcal{G}}_h, b_e\bm{e}_1 \rangle_{-1,1}.
\end{align}
In the view of Lemma \ref{Lemama}, the following holds
\begin{align*}
h^2_e \|\hat{\sigma_1} (\bm{u}_h)\|^2_{\bm{L}^{\infty}(e)} \lesssim h_e \|\hat{\sigma_1} (\bm{u}_h)\|^2_{\bm{L}^{2}(e)} \lesssim h_e \int_e \hat{\sigma_1} (\bm{u}_h) \hat{\sigma_1} (\bm{u}_h) b_e~ds.
\end{align*}
Inserting \eqref{eqeq}, we get
\begin{align*}
h^2_e \|\hat{\sigma_1} (\bm{u}_h)\|^2_{\bm{L}^{\infty}(e)} = h_e \Big(-\langle \bm{\mathcal{G}}_h, \hat{\sigma_1} (\bm{u}_h) b_e \bm{e}_1 \rangle_{-1,1} & + \int_{\cT_e} s_1(\bm{u}_h) \hat{\sigma_1} (\bm{u}_h) b_e~ds  \\ &+  \int_e ( \hat{\bm{\sigma}}(\bm{u}_h)- \bm{g}) \cdot  \hat{\sigma_1} (\bm{u}_h) b_e \bm{e}_1 ~ds \Big ).
\end{align*}
Using Lemma \ref{Lemama}, H\"older's inequality and the structure of $b_e$, we get
\begin{align*}
h^2_e \|\hat{\sigma_1} (\bm{u}_h)\|^2_{\bm{L}^{\infty}(e)}	& \lesssim h_e \Big( \|\bm{\mathcal{G}}_h\|_{-2,\infty,T} \|\hat{\sigma_1} (\bm{u}_h) b_e \|_{\bm{W}^{2,1}(T)} \\ & \hspace{3cm}+\bigg \{\|\bm{s}(\bm{u}_h)\|_{\bm{L}^{\infty}(T)} + \|\hat{\bm{\sigma}}(\bm{u}_h)- \bm{g} \|_{\bm{L}^{\infty}(e)}   \bigg\} 
\|\hat{\sigma_1} (\bm{u}_h) b_e \|_{\bm{L}^{1}(T)} \Big) \notag \\ & \lesssim h_e \Big( \|\bm{\mathcal{G}}_h\|_{-2,\infty,T}+  h^2_e \|\bm{s}(\bm{u}_h)\|_{\bm{L}^{\infty}(T)} +h_e\|\hat{\bm{\sigma}}(\bm{u}_h)- \bm{g} \|_{\bm{L}^{\infty}(e)}   \Big) \|\hat{\sigma_1} (\bm{u}_h) b_e \|_{\bm{L}^{\infty}(T)}  \notag \\ & \lesssim h_e \Big( \|\bm{\mathcal{G}}_h\|_{-2,\infty,T}  + h^2_T \|\bm{s}(\bm{u}_h)\|_{\bm{L}^{\infty}(T)}  + h_e\|\hat{\bm{\sigma}}(\bm{u}_h)- \bm{g} \|_{\bm{L}^{\infty}(e)}  \Big) \|\hat{\sigma_1} (\bm{u}_h) \|_{\bm{L}^{\infty}(e)} .
\end{align*}
Finally, the proof of \eqref{eff5} follows using equations \eqref{eff1} and \eqref{eff3}.
\end{proof}
\section{Numerical Results} \label{sec7}
\noindent
In this section, we employ the error estimator $\eta_h$ (defined in equation \eqref{est}) to solve a variety of contact problems on adaptive meshes. For the adaptive refinement, we use the algorithm based on the following four steps.
\begin{equation*}
{\bf SOLVE}\rightarrow  {\bf ESTIMATE} \rightarrow {\bf
	MARK}\rightarrow {\bf REFINE}
\end{equation*}
In the {\bf SOLVE} step, we solve the discrete inequality (equation \eqref{eq:DVI}) using the primal-dual active set algorithm  \cite{hueber2005priori}. In the {\bf ESTIMATE} step, we evaluate  a posteriori error estimator $\eta_h$ on each element $T \in \cT_h$ where the factor $C_0$ of Theorem \ref{thm:rel} is practically replaced by $C_0 = 0.45$. This suitable choice is consistent with respect to \eqref{rel1} in all the experiments with reasonable shape-regularity and moderate $h_{min}$. To compute the supremum norm, functions are evaluated at {the quadrature points}. In the {\bf MARK} step, the adaptive refinement is based on the maximum norm criterion \cite{verfurth1996review}, which
seems adequate for error control in the supremum norm.
Below, we consider two model contact problems in which the first example is chosen such that the continuous solution $\bm{u}$ is known and smooth and {the continuous solution for the second example is unknown}.
\begin{example} \label{ex1}
Let $\Omega = (0,1)^2$ and we assume the top of our domain $\Omega$ is fixed. The given data is as follows:
	\begin{align*}
	\Gamma_C= (0,1) \times \{0\},~ \Gamma_D= (0,1) \times \{1\}~\text{and}~ \Gamma_N= \{0,1\} \times (0,1).
	\end{align*}
	Let $\zeta=\mu=1$ and the given data $\bm{g}$ and $\bm{f}$ are chosen such that the continuous solution takes the form $\bm{u}:=(u_1,u_2)= (y^2(y-1)),(e^y(1-y)y(x-2))$. 
	In this case, we note that $\bm{n}=(0,-1)$ on $\Gamma_C$, hence, the error estimator defined in \eqref{est} will modify to
	\begin{align}
	\eta_h&= {l_h} \Big(\sum_{i=1}^{5} \eta_i \Big) + \|(-u_{h,2}-\chi)^{+}\|_{L^{\infty}(\Gamma_C)} + \|(\chi+u_{h,2})^+\|_{L^{\infty}(\Lambda_h^{C})}, \\ 
	\text{with}\notag\\
	\Lambda_h^{C} &:= \{\Gamma_{p,C}~:~p \in \cN^C_h \cup \cM^C_h~\text{such that}~ {	\langle  \bm{\lambda}_h, \psi_p\bm{e}_2 \rangle_h <0}  \}.\notag
	\end{align}
	 \Cref{3(a)} displays the behavior of the error $\|\bm{u}-\bm{u}_h\|_{\bm{L}^{\infty}(\Omega)}$ and estimator $\eta_h$ versus the number of degrees of freedom (NDF) in the log-log plot. It is evident that the error $\|\bm{u}-\bm{u}_h\|_{\bm{L}^{\infty}(\Omega)}$ and estimator $\eta_h$ converge with the optimal rate (1/(NDF)$^{3/2}$). The efficiency index which is depicted in \Cref{3(b)} indicates the efficiency of the error estimator. Here, the term $\|(\chi+u_{h,2})^+\|_{L^{\infty}(\Lambda_h^{C})}$ is zero since the inactive region on $\Gamma_C$ is empty owing to $\chi=0$. Further, noting that $\chi=\chi_h$ on $\Gamma_C$, the quantity $\|(-u_{h,2}-\chi)^{+}\|_{L^{\infty}(\Gamma_C)}$ vanishes on $\Gamma_C$.  The plot of contributions of individual estimator $\eta_i,~1\leq i \leq 5$ is depicted in \Cref{5(a)}.
\end{example}
\setcounter{figure}{1}
\renewcommand{\thefigure}{\arabic{figure}}
\begin{figure}
		\subfigure[Error and Estimator\label{3(a)}]{\includegraphics[width=8cm]{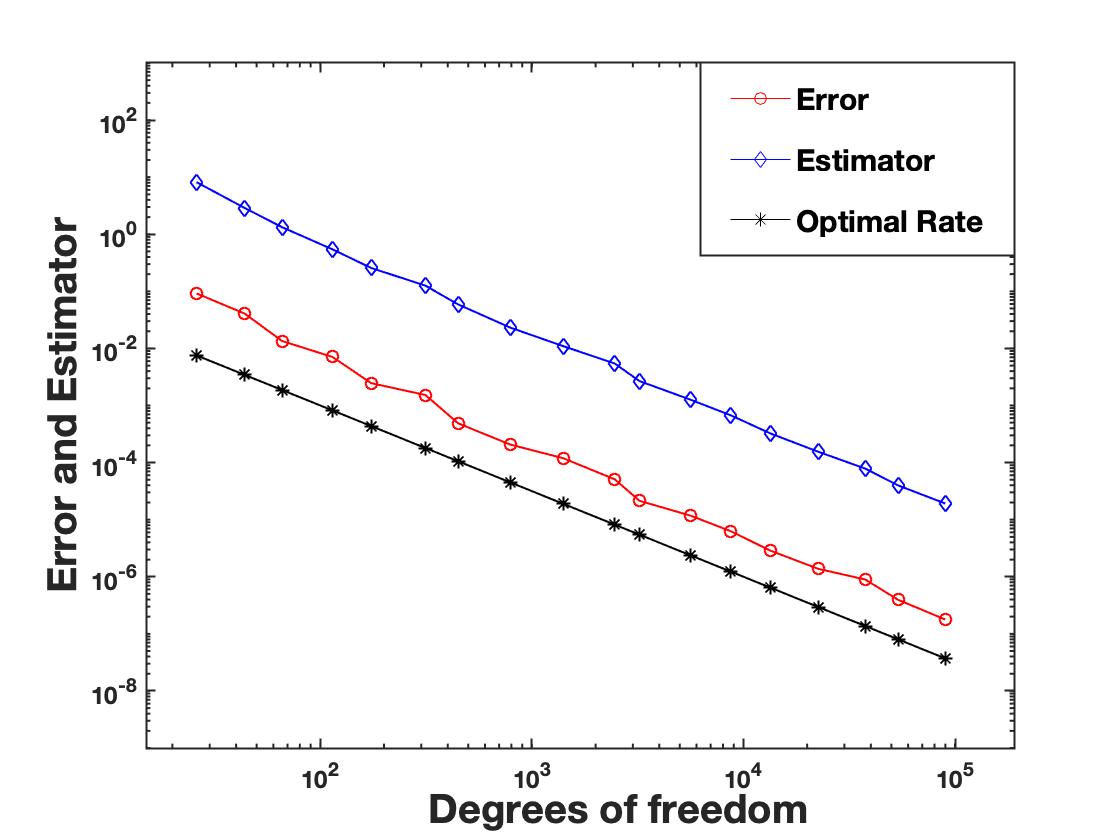}}
			\subfigure[Efficiency Index\label{3(b)}]{\includegraphics[width=8cm]{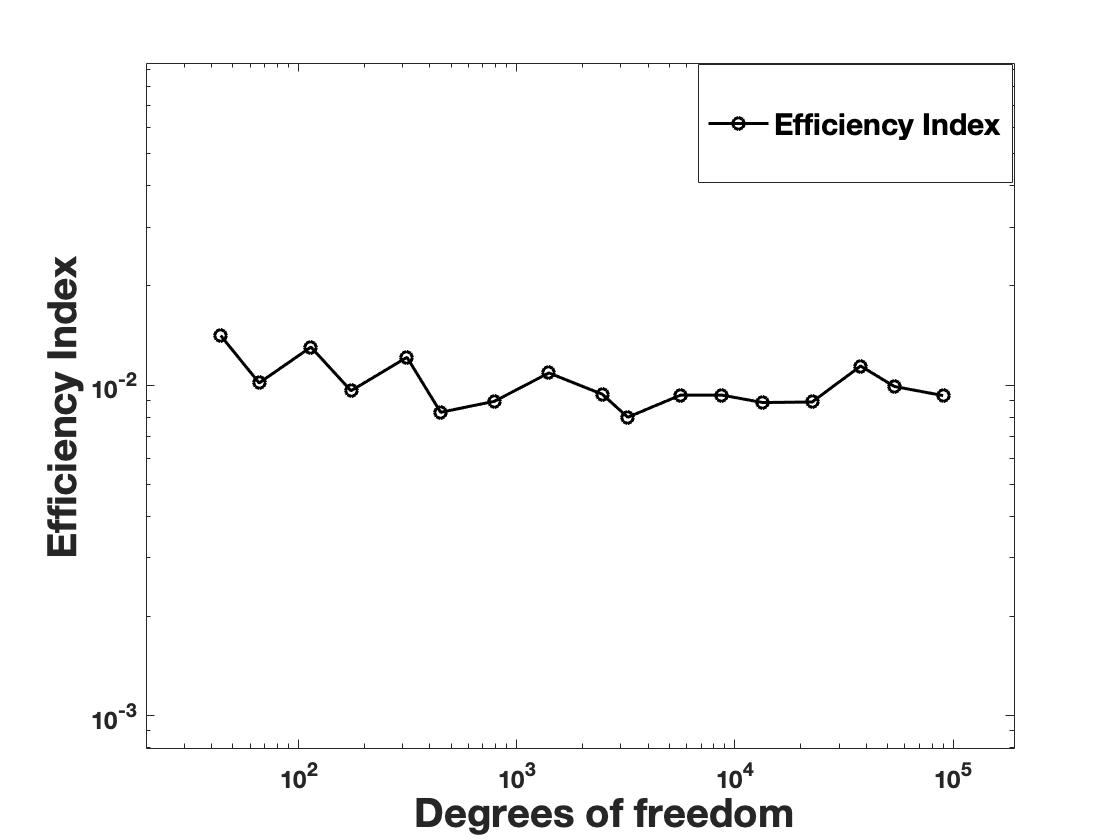}}
\caption{Plot of error, estimator and efficiency index for Example 7.1}
\label{3}
\end{figure}
\begin{figure}
	\subfigure[Estimator\label{4(a)}]{\includegraphics[width=8cm]{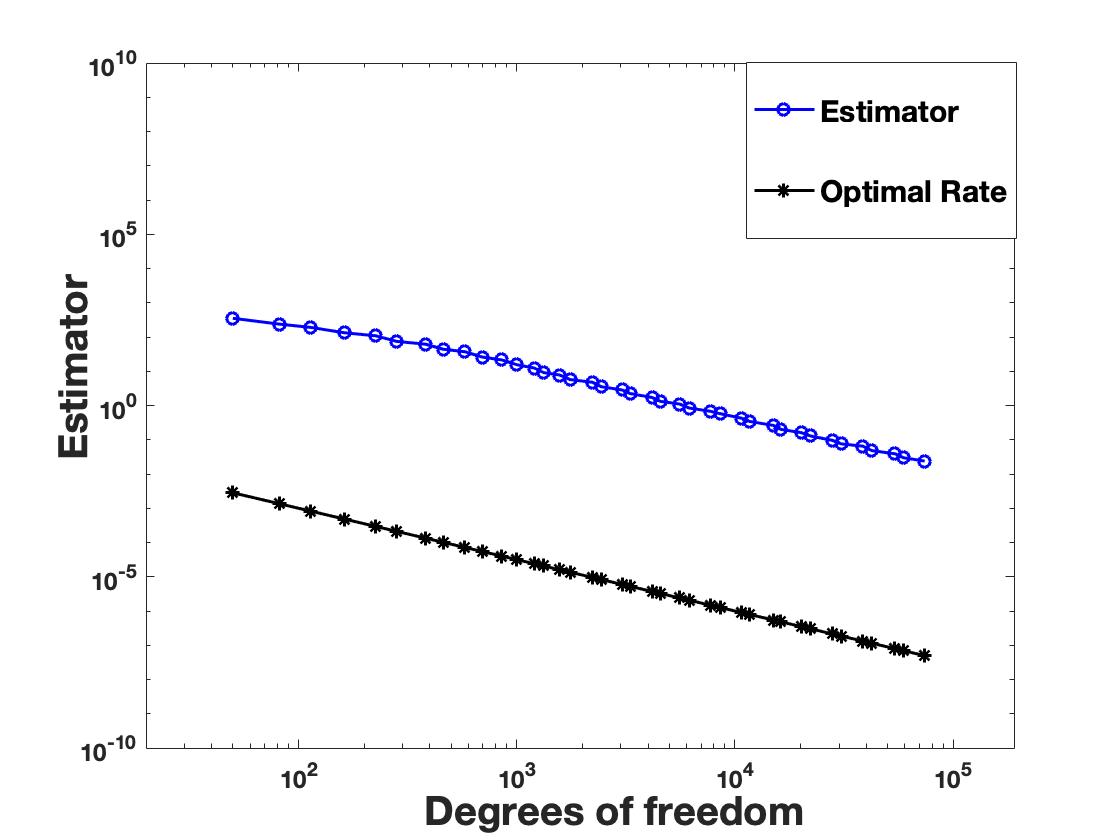}}
	\subfigure[Adaptive Mesh \label{4(b)}]{\includegraphics[width=8cm]{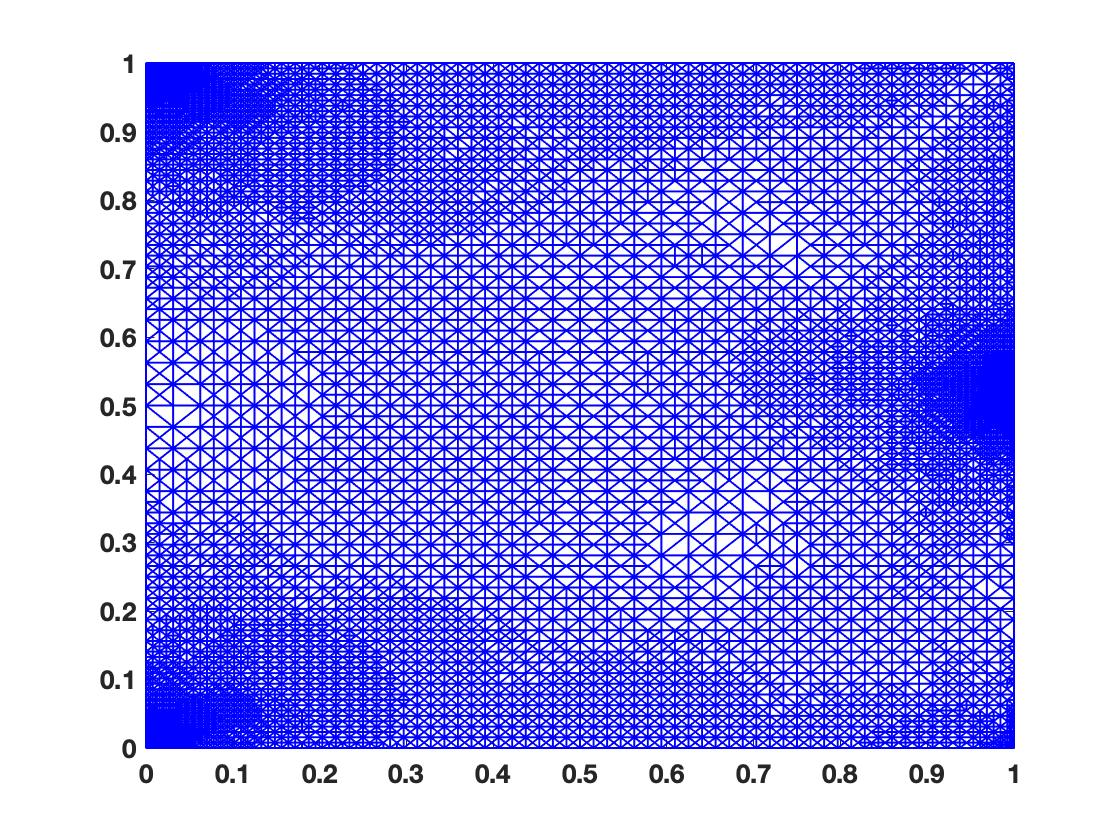}}
	\caption{Estimator and Adaptive mesh for Example 7.2}
	\label{4}
\end{figure}
\begin{figure}
	\subfigure[Example 7.1 \label{5(a)}]{\includegraphics[width=8cm]{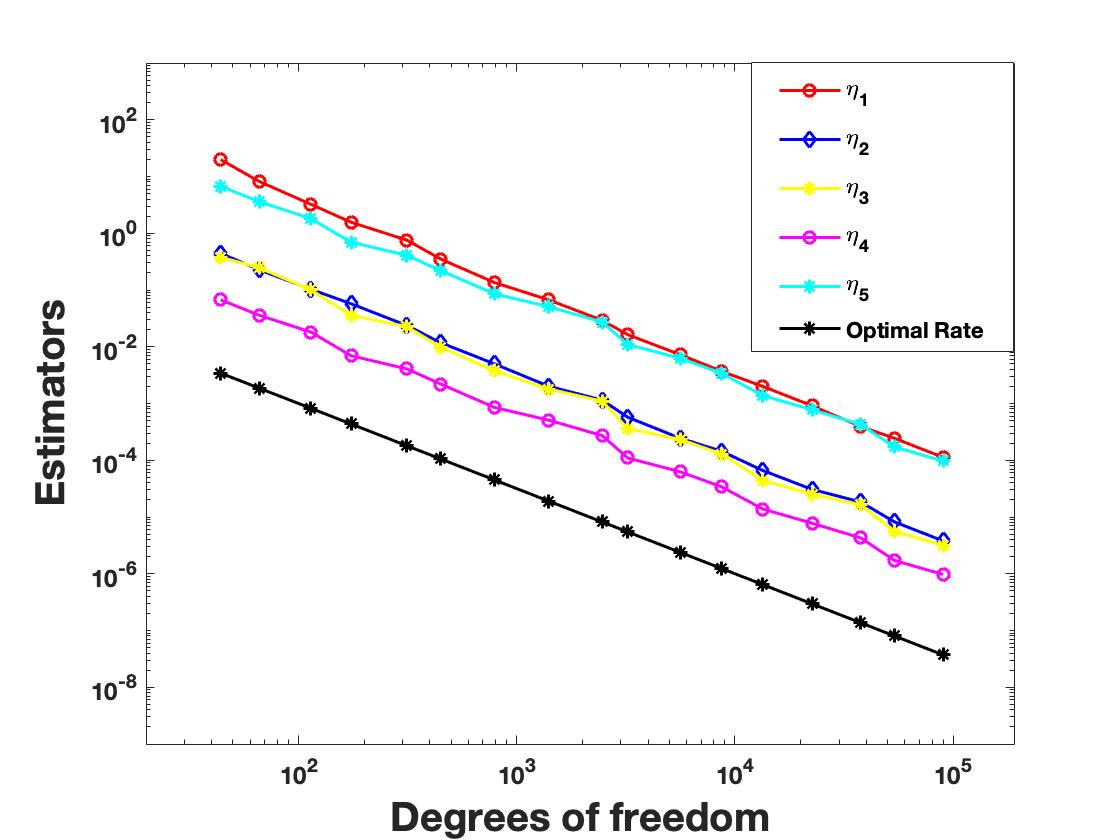}}
		\subfigure[Example 7.2 \label{5(b)}]{\includegraphics[width=8cm]{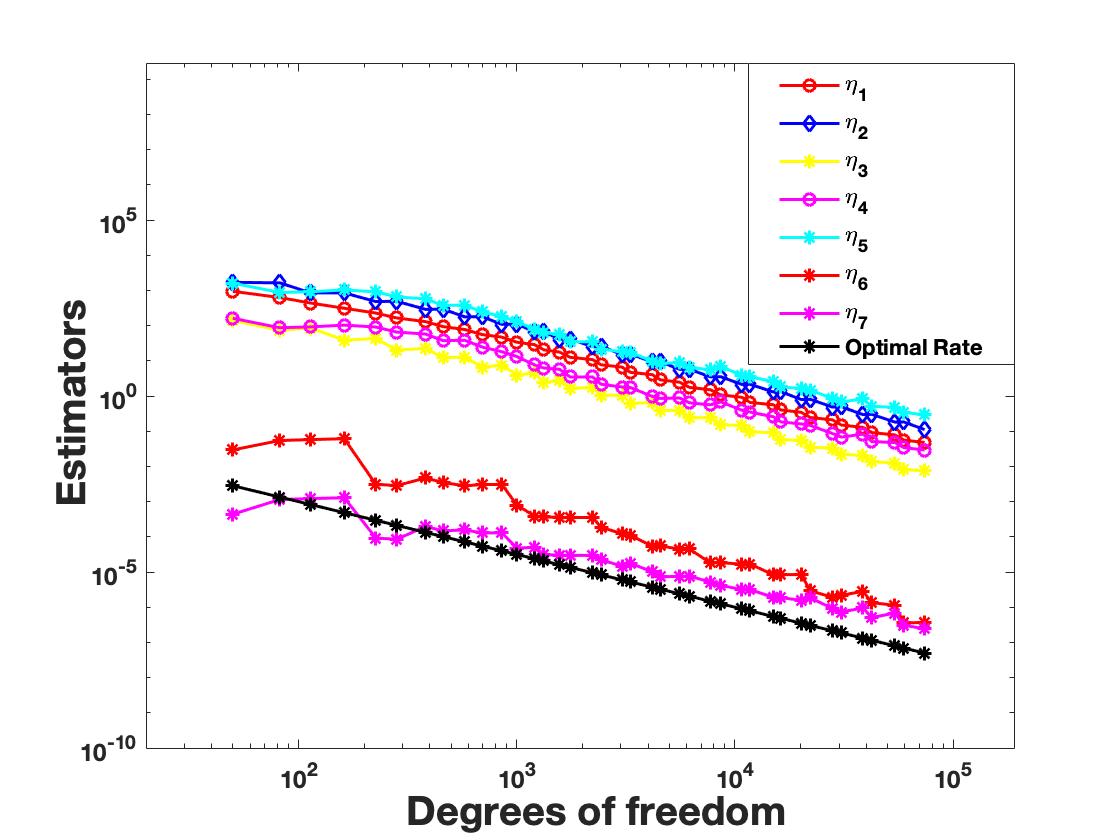}}
		\caption{Plot of individual estimators}
		\label{5}
	\end{figure}
\begin{example} \label{ex2}[(Contact with a rigid wedge \cite{walloth2019reliable})] \label{ex4}
		In this example, we consider the deformation of the $\Omega=(0,1) \times (0,1)$ which is pushed along the $x$ direction towards the non zero obstacle $\chi(y)=-0.2+0.5|y-0.5|$. Let us consider
		\begin{align*}
		\Gamma_C=  \{1\}\times (0,1),~ \Gamma_D= \{0\} \times (0,1)~\text{and}~ \Gamma_N= (0,1) \times \{1\} \cup (0,1) \times \{0\} .
		\end{align*} The Young's modulus and Poisson's ratio are assumed to be $E=500$ and $\nu=0.3$, respectively. Let $\bm{g}=\bm{f}= \bm{0}$ and the non homogeneous Dirichlet data is $\bm{u}=(0.1,0)$. The plot of the error estimator $\eta_h$ is shown in \Cref{4(a)} with logarithmic scales on both axes and the convergence behaviour of estimator  $\eta_i, 1\leq i \leq5$ together with $\eta_6=\|(u_{h,1}-\chi)^{+}\|_{L^{\infty}(\Gamma_C)}$ and $\eta_7=\|(\chi-u_{h,1})^+\|_{L^{\infty}(\Lambda_h^{C})}$ is illustrated in \Cref{5(b)}. We note that the full estimator $\eta_h$ and the individual estimators converge optimally. In \Cref{4(b)} the  adaptive mesh at level 20 is displayed and as expected, there is more refinement around the free boundary region and near the intersection corners of Dirichlet and Neumann boundaries.
	\end{example}
\bibliographystyle{unsrt}
\bibliography{Quad_V2}

\begin{thebibliography}{10}

\bibitem{kikuchi1988contact}
Noboru Kikuchi and John~Tinsley Oden.
\newblock {\em {Contact problems in elasticity: a study of variational inequalities and finite element methods}}.
\newblock SIAM, 1988.

\bibitem{walloth2012}
Mirjam Walloth.
\newblock {\em {Adaptive numerical simulation of contact problems: resolving local effects at the contact boundary in space and time. PhD thesis, Rheinische Friedrich-Wilhelms-Universität Bonn}}.
\newblock 2012.

\bibitem{hild2002quadratic}
Patrick Hild and Patrick Laborde.
\newblock {Quadratic finite element methods for unilateral contact problems}.
\newblock {\em Applied Numerical Mathematics}, 41(3):401--421, 2002.

\bibitem{glowinski2008lectures}
Roland Glowinski.
\newblock {\em {Lectures on numerical methods for non-linear variational problems}}.
\newblock Springer Science \& Business Media, 2008.

\bibitem{Signorini:1977}
D.~Kinderlehrer.
\newblock {{Remarks about Signorini’s problem in linear elasticity}}.
\newblock {\em Ann. Scuola Norm. Sup. Pisa Cl. Sci.}, 8(4):605--645, 1981.

\bibitem{Caffarelli:1979}
L.~A. Caffarelli.
\newblock {{Further regularity for the Signorini problem}}.
\newblock {\em Comm. Partial Differential Equations}, 4(9):1067--1075, 1979.

\bibitem{brezzi1977error}
Franco Brezzi, William~W Hager, and Pierre-Arnaud Raviart.
\newblock {Error estimates for the finite element solution of variational inequalities}.
\newblock {\em Numerische Mathematik}, 28(4):431--443, 1977.

\bibitem{scarpini1977error}
Francesco Scarpini and Maria~Agostina Vivaldi.
\newblock Error estimates for the approximation of some unilateral problems.
\newblock {\em RAIRO Analyse num{\'e}rique}, 11(2):197--208, 1977.

\bibitem{ciarlet2002finite}
Philippe~G Ciarlet.
\newblock {\em {The finite element method for elliptic problems}}.
\newblock SIAM, 2002.

\bibitem{belhachmi2003quadratic}
Zakaria Belhachmi and F~Belgacem.
\newblock {Quadratic finite element approximation of the Signorini problem}.
\newblock {\em Mathematics of Computation}, 72(241):83--104, 2003.

\bibitem{veeser2001efficient}
Andreas Veeser.
\newblock {Efficient and reliable a posteriori error estimators for elliptic obstacle problems}.
\newblock {\em SIAM Journal on Numerical Analysis}, 39(1):146--167, 2001.

\bibitem{weiss2009posteriori}
Alexander Weiss and Barbara~I Wohlmuth.
\newblock {A posteriori error estimator and error control for contact problems}.
\newblock {\em Mathematics of Computation}, 78(267):1237--1267, 2009.

\bibitem{krause2015efficient}
Rolf Krause, Andreas Veeser, and Mirjam Walloth.
\newblock {An efficient and reliable residual-type a posteriori error estimator for the Signorini problem}.
\newblock {\em Numerische Mathematik}, 130(1):151--197, 2015.

\bibitem{gudi2016posteriori}
Thirupathi Gudi and Kamana Porwal.
\newblock {A posteriori error estimates of discontinuous Galerkin methods for the Signorini problem}.
\newblock {\em Journal of Computational and Applied Mathematics}, 292:257--278, 2016.

\bibitem{walloth2019reliable}
Mirjam Walloth.
\newblock {A reliable, efficient and localized error estimator for a discontinuous Galerkin method for the Signorini problem}.
\newblock {\em Applied Numerical Mathematics}, 135:276--296, 2019.

\bibitem{KP:2022:QuadSignorini}
Rohit Khandelwal, Kamana Porwal, and Tanvi Wadhawan.
\newblock {Adaptive quadratic finite element method for the unilateral contact problem}.
\newblock {\em Submitted}.

\bibitem{nochetto1995pointwise}
Ricardo~H Nochetto.
\newblock {Pointwise a posteriori error estimates for elliptic problems on highly graded meshes}.
\newblock {\em Mathematics of Computation}, 64(209):1--22, 1995.

\bibitem{dari1999maximum}
Enzo Dari, Ricardo~G Dur{\'a}n, and Claudio Padra.
\newblock {Maximum norm error estimators for three-dimensional elliptic problems}.
\newblock {\em SIAM Journal on Numerical Analysis}, 37(2):683--700, 1999.

\bibitem{demlow2012pointwise}
Alan Demlow and Emmanuil~H Georgoulis.
\newblock {Pointwise a posteriori error control for discontinuous Galerkin methods for elliptic problems}.
\newblock {\em SIAM Journal on Numerical Analysis}, 50(5):2159--2181, 2012.

\bibitem{nitsche1977convergence}
Joachim Nitsche.
\newblock {$L^{\infty}$-convergence of finite element approximations}.
\newblock In {\em Mathematical aspects of finite element methods}, pages 261--274. Springer, 1977.

\bibitem{baiocchi1977estimations}
C~Baiocchi.
\newblock {Estimations d’erreur dans $L^{\infty}$ pour les in{\'e}quations {\`a} obstacle}.
\newblock In {\em Mathematical Aspects of Finite Element Methods}, pages 27--34. Springer, 1977.

\bibitem{nochetto2003pointwise}
Ricardo~H Nochetto, Kunibert~G Siebert, and Andreas Veeser.
\newblock {Pointwise a posteriori error control for elliptic obstacle problems}.
\newblock {\em Numerische Mathematik}, 95(1):163--195, 2003.

\bibitem{nochetto2005fully}
Ricardo~H Nochetto, Kunibert~G Siebert, and Andreas Veeser.
\newblock {Fully localized a posteriori error estimators and barrier sets for contact problems}.
\newblock {\em SIAM Journal on Numerical Analysis}, 42(5):2118--2135, 2005.

\bibitem{BGP:2022:Obstacle}
B.~Ayuso~de Dios, Thirupathi Gudi, and Kamana Porwal.
\newblock {Pointwise a posteriori error analysis of a discontinuous Galerkin method for the elliptic obstacle problem}.
\newblock {\em Accepted in IMA Journal of Numerical Analysis}.

\bibitem{khandelwal2022pointwiseq}
Rohit Khandelwal and Kamana Porwal.
\newblock {Pointwise a posteriori error analysis of quadratic finite element method for the elliptic obstacle problem}.
\newblock {\em Accepted for Publication in Journal of Computational and Applied Mathematics}, 2022.

\bibitem{KP:2021:Signorini}
Rohit Khandelwal and Kamana Porwal.
\newblock {Pointwise a posteriori error analysis of a finite element method for the Signorini problem}.
\newblock {\em Journal of Scientific Computing}, 91(2):1--34, 2022.

\bibitem{dolzmann1995estimates}
Georg Dolzmann and Stefan M{\"u}ller.
\newblock {Estimates for Green's matrices of elliptic systems by $L^p$ theory}.
\newblock {\em Manuscripta Mathematica}, 88(1):261--273, 1995.

\bibitem{moon2007posteriori}
Kyoung-Sook Moon, Ricardo~H Nochetto, Tobias Von~Petersdorff, and Chen-song Zhang.
\newblock {A posteriori error analysis for parabolic variational inequalities}.
\newblock {\em ESAIM: Mathematical Modelling and Numerical Analysis}, 41(3):485--511, 2007.

\bibitem{nochetto2006pointwise}
Ricardo~H Nochetto, Alfred Schmidt, Kunibert~G Siebert, and Andreas Veeser.
\newblock Pointwise a posteriori error estimates for monotone semi-linear equations.
\newblock {\em Numerische Mathematik}, 104(4):515--538, 2006.

\bibitem{glowinski1980numerical}
Roland Glowinski.
\newblock {\em {Numerical methods for nonlinear variational problems}}.
\newblock Tata Institute of Fundamental Research, 1980.

\bibitem{hofmann2007green}
Steve Hofmann and Seick Kim.
\newblock {The Green's function estimates for strongly elliptic systems of second order}.
\newblock {\em Manuscripta Mathematica}, 124(2):139--172, 2007.

\bibitem{dong2009green}
Hongjie Dong and Seick Kim.
\newblock {Green's matrices of second order elliptic systems with measurable coefficients in two dimensional domains}.
\newblock {\em Transactions of the American Mathematical Society}, 361(6):3303--3323, 2009.

\bibitem{brenner2007mathematical}
Susanne Brenner and Ridgway Scott.
\newblock {\em {The mathematical theory of finite element methods}}, volume~15.
\newblock Springer Science \& Business Media, 2007.

\bibitem{fierro2003posteriori}
Francesca Fierro and Andreas Veeser.
\newblock {A posteriori error estimators for regularized total variation of characteristic functions}.
\newblock {\em SIAM Journal on Numerical Analysis}, 41(6):2032--2055, 2003.

\bibitem{verfurth1996review}
R{\"u}diger Verf{\"u}rth.
\newblock {\em {A review of a posteriori error estimation and Adaptive Mesh-Refinement Techniques}}.
\newblock Wiley \& Teubner, Citeseer, 1996.

\bibitem{evans2010partial}
Lawrence~C Evans.
\newblock {\em Partial differential equations}, volume~19.
\newblock American Mathematical Soc., 2010.

\bibitem{ambrosio2015lecture}
Luigi Ambrosio.
\newblock {Lecture notes on elliptic partial differential equations}.
\newblock {\em Unpublished lecture notes. Scuola Normale Superiore di Pisa}, 30, 2015.

\bibitem{kashiwabara2020pointwise}
Takahito Kashiwabara and Tomoya Kemmochi.
\newblock {Pointwise error estimates of linear finite element method for Neumann boundary value problems in a smooth domain}.
\newblock {\em Numerische Mathematik}, 144(3):553--584, 2020.

\bibitem{hueber2005priori}
Stefan H{\"u}eber, Michael Mair, and Barbara~I Wohlmuth.
\newblock {A priori error estimates and an inexact primal-dual active set strategy for linear and quadratic finite elements applied to multibody contact problems}.
\newblock {\em Applied Numerical Mathematics}, 54(3-4):555--576, 2005.

\end{thebibliography}
\end{document}